\documentclass[kpfonts]{patmorin}
\usepackage{pat}
\usepackage{paralist}
\usepackage{dsfont}  % for \mathds{A}
\usepackage[utf8x]{inputenc}
\usepackage{skull}
\usepackage{graphicx}
\usepackage[noend]{algorithmic}
\usepackage[normalem]{ulem}
\usepackage{cancel}
\usepackage{enumitem}
\usepackage{todonotes}

\usepackage[longnamesfirst,numbers,sort&compress]{natbib}

\usepackage[mathlines]{lineno}
\newcommand*\patchAmsMathEnvironmentForLineno[1]{%
 \expandafter\let\csname old#1\expandafter\endcsname\csname #1\endcsname
 \expandafter\let\csname oldend#1\expandafter\endcsname\csname end#1\endcsname
 \renewenvironment{#1}%
    {\linenomath\csname old#1\endcsname}%
    {\csname oldend#1\endcsname\endlinenomath}}%
\newcommand*\patchBothAmsMathEnvironmentsForLineno[1]{%
 \patchAmsMathEnvironmentForLineno{#1}%
 \patchAmsMathEnvironmentForLineno{#1*}}%
\AtBeginDocument{%
\patchBothAmsMathEnvironmentsForLineno{equation}%
\patchBothAmsMathEnvironmentsForLineno{align}%
\patchBothAmsMathEnvironmentsForLineno{flalign}%
\patchBothAmsMathEnvironmentsForLineno{alignat}%
\patchBothAmsMathEnvironmentsForLineno{gather}%
\patchBothAmsMathEnvironmentsForLineno{multline}%
}

\definecolor{brightmaroon}{rgb}{0.76, 0.13, 0.28}
\definecolor{linkblue}{rgb}{0, 0.337, 0.227}
\newcommand{\defin}[1]{\emph{\color{brightmaroon}#1}}
\DeclareMathOperator{\diam}{diam}
\DeclareMathOperator{\tw}{tw}
\DeclareMathOperator{\td}{td}
\DeclareMathOperator{\stw}{stw}

\DeclareMathOperator{\pw}{pw}

\title{\MakeUppercase{Asymptotically Optimal Vertex Ranking of Planar Graphs}\thanks{This research was partly funded by NSERC.}}
\author{Prosenjit Bose%
    \thanks{School of Computer Science, Carleton University}\quad
    Vida Dujmović%
    \thanks{Department of Computer Science and Electrical Engineering, University of Ottawa}\quad
    Mehrnoosh Javarsineh\footnotemark[2]  \quad
    Pat Morin\footnotemark[2]}

\newcommand{\rn}[1]{\chi_{\operatorname{#1-vr}}}
\newcommand{\irn}{\rn{\infty}}
\newcommand{\trn}{\rn{2}}
\newcommand{\lrn}{\rn{\ell}}
\newcommand{\dtcn}{\bar{\chi}_2}
\newcommand{\dlcn}{\bar{\chi}_\ell}
\newcommand{\scn}{\chi_{\star}}

\newtheorem{othertheorem}{Theorem}

\crefname{othertheorem}{Theorem}{Theorem}

\newtheoremstyle{named}{}{}{\itshape}{}{\bfseries}{.}{.5em}{#3}
\theoremstyle{named}
\newtheorem*{namedtheorem}{Unused}

\newcommand{\weirdref}[2]{\cref{#1}#2}
\newcommand{\weirdlabel}[2]{\label{#1-#1}}

\begin{document}
\begin{titlepage}
\maketitle

\begin{abstract}
  A (vertex) $\ell$-ranking is a colouring $\varphi:V(G)\to\N$ of the vertices of a graph $G$ with integer colours so that for any path $u_0,\ldots,u_p$ of length at most $\ell$, $\varphi(u_0)\neq\varphi(u_p)$ or $\varphi(u_0)<\max\{\varphi(u_0),\ldots,\varphi(u_p)\}$.  We show that, for any fixed integer $\ell\ge 2$, every $n$-vertex planar graph has an $\ell$-ranking using $O(\log n/\log\log\log n)$ colours and this is tight even when $\ell=2$; for infinitely many values of $n$, there are $n$-vertex planar graphs, for which any 2-ranking requires $\Omega(\log n/\log\log\log n)$ colours.  This result also extends to bounded genus graphs.

  In developing this proof we obtain optimal bounds on the number of colours needed for $\ell$-ranking graphs of treewidth $t$ and graphs of simple treewidth $t$.  These upper bounds are constructive and give $O(n)$-time algorithms.  Additional results that come from our techniques include new sublogarithmic upper bounds on the number of colours needed for $\ell$-rankings of apex minor-free graphs and $k$-planar graphs.
\end{abstract}
\end{titlepage}

% \tableofcontents
%
% \newpage
\pagenumbering{arabic}

\section{Introduction}

% A sequence of integers $\phi_0,\ldots,\phi_p$ is \defin{ranked} if $\max\{\phi_0,\ldots,\phi_p\}$ occurs exactly once in $\phi_0,\ldots,\phi_p$.

A \defin{colouring} $\varphi:V(G)\to \N$ of a graph $G$ is a \defin{(vertex) $\ell$-ranking} of $G$ if, for every connected non-empty subgraph $X\subseteq G$ of diameter\footnote{The \defin{length} of a path $u_0,\ldots,u_p$ is the number, $p$, of edges in the path. A path is \defin{trivial} if its length is 0 and non-trivial otherwise. The \defin{distance} between two vertices $v$ and $w$ in a graph is the length of a shortest path that contains $v$ and $w$, or $\infty$ if $v$ and $w$ are in different components of $G$. The \defin{diameter} of a graph $G$ is the maximum distance between any pair of vertices in $G$.}  at most $\ell$, there exists exactly one vertex $v\in V(X)$ such that $\varphi(v)=\max\{\varphi(w):w\in V(X)\}$.  The \defin{$\ell$-ranking number} $\lrn(G)$ of $G$ is the minimum integer $k$ such that $G$ has an $\ell$-ranking $\varphi:V(G)\to \{1,\ldots,k\}$.  Note that, for any $\ell\ge 1$ any $\ell$-ranking of $G$ is a proper colouring\footnote{A colouring $\varphi:V(G)\to \N$ is \defin{proper} if, for each edge $vw\in E(G)$, $\varphi(v)\neq\varphi(w)$ and the \defin{chromatic number}, $\chi(G)$, of $G$ is the minimum integer $k$ such that there exists a proper colouring $\varphi:V(G)\to\{1,\ldots,k\}$ of $G$.} of $G$, so $\chi(G)\le \lrn(G)$, and any proper colouring of $G$ is a 1-ranking of $G$, so $\chi(G)=\rn{1}(G)$.

Besides the case $\ell=1$, two cases have received special attention: An $\infty$-ranking is called a \defin{vertex ranking} or \defin{ordered colouring}. The parameter $\irn(G)$ is called the \defin{vertex ranking number} of $G$ and is equal to the \defin{treedepth} of $G$ which is equal to the \defin{centered chromatic number} of $G$ \cite{nesetril.ossona:tree-depth}.
% This parameter has applications to matrix factorization \cite{bodlaender.gilbert.ea:approximating,duff.reid:multifrontal,liu:role,dereniowski.kubale:cholesky}, VLSI layout \cite{leiserson:area,sen.deng.ea:on}, and the analysis of online algorithms \cite{even.smorodinsky:hitting}, and are fundamental to theory of sparsity \cite{nesetril.ossona:sparsity}.
The parameter $\lrn(G)$ for finite $\ell\ge 2$ appears implicitly in a dynamic programming algorithm of \citet{deogun.kloks.ea:on} for computing $\irn(G)$ when $G$ is a $d$-trapezoid graph.  The case $\ell=2$ has also received special attention \cite{almeter.demircan.ea:graph,karpas.neiman.ea:on,shalu.antony:complexity}. A $2$-ranking is called a \defin{unique-superior colouring} by \citet{karpas.neiman.ea:on} who prove the following result:

\setcounter{othertheorem}{19}
\begin{othertheorem}[\cite{karpas.neiman.ea:on}]\label{trees}
    For every $n$-vertex tree $T$, $\trn(T)\in O(\log n/\log\log n)$ and this is asymptotically optimal: for infinitely many values of $n$, there exists an $n$-vertex tree $T$ with $\trn(T)\in\Omega(\log n/\log\log n)$.
\end{othertheorem}

The same authors prove the following result for planar graphs:

\setcounter{othertheorem}{15}
\begin{othertheorem}[\cite{karpas.neiman.ea:on}]\label{planar-graphs}
    For every integer $\ell$ and every $n$-vertex planar graph $G$, $\lrn(G)\in O(\ell\log n)$.
\end{othertheorem}

Since every tree is a planar graph and no better lower bound is known for planar graphs, this leaves an obvious question:  Which is the correct bound for 2-ranking $n$-vertex planar graphs, $\log n$ or $\log n/\log\log n$?  As it turns out, the strange truth is somewhere in between.  Let $\log x :=\ln x$ denote the natural logarithm of $x$ and define $\log^{(0)}x:=x$ and, for any integer $i>0$, let $\log^{(i)}x:=\log(\log^{(i-1)} x)$. We prove:\footnote{Refined versions of \cref{planar} and of the upcoming \cref{simple-t-trees,t-trees,bounded-genus,meta} that describes the dependence of $\lrn(G)$ on $\ell$ are presented in \cref{dependence-on-ell}.}

\begin{thm}\label{planar}
    For any fixed integer $\ell\ge 2$, every $n$-vertex planar graph $G$ has $\lrn(G)\in O(\log n/\log^{(3)} n)$ and this is asymptotically optimal: for infinitely many values of $n$, there exists an $n$-vertex planar graph $G$ with $\trn(G)\in \Omega(\log n/\log^{(3)} n)$
\end{thm}

Our proof of the upper bound in \cref{planar} makes use of a recent \defin{product structure theorem} of \citet{dujmovic.joret.ea:planar} which states that every planar graph $G$ is a subgraph of $H\boxtimes K_3\boxtimes P$ where $H$ is a planar graph of treewidth at most $3$, $K_3$ is a 3-cycle, $P$ is a path, and $\boxtimes$ denotes the strong graph product.\footnote{Definitions of $t$-trees, simple $t$-trees, treewidth, simple treewidth, and strong graph product appear later, in \cref{basics}.}  To apply this theorem, we prove the following result:

\begin{thm}\label{simple-t-trees}
    For any fixed integers $\ell\ge 2$ and $t\ge 1$, every $n$-vertex graph $H$ of simple treewidth at most $t$ has $\lrn(H) \in O(\log n/\log^{(t)}n)$ and this is asymptotically optimal: for any fixed integer $t\ge 1$ and infinitely many values of $n$, there exists an $n$-vertex graph $H$ of simple treewidth $t$ that has $\trn(H)\in\Omega(\log n/\log^{(t)} n)$.
\end{thm}

The lower bound in \cref{simple-t-trees} immediately implies the lower bound in \cref{planar} because a graph has simple treewidth at most 3 if and only if it is planar and has treewidth at most 3. Therefore, the lower bound in \cref{simple-t-trees} shows the existence of $n$-vertex planar graphs $H$ with $\trn(H)\in\Omega(\log n/\log^{(3)} n)$.

To obtain the upper bound in \cref{planar}, we apply the upper bound in \cref{simple-t-trees} to the graph $H$ that appears in the product structure theorem along with a simple lemma which shows that, for any two graphs $G_1$ and $G_2$, $\lrn(G_1\boxtimes G_2)\le \lrn(G_1)\cdot\dlcn(G_2)$ where $\dlcn(G_2)$ is the \defin{distance-$\ell$ colouring number} of $G_2$;  the minimum number of colours needed to colour $G_2$ so that the endpoints of each non-trivial path of length at most $\ell$ have different colours.  It is easy to see that $\dlcn(K_3\times P)\le 3(\ell+1)$, so $\lrn(H\boxtimes K_3\boxtimes P)\le 3(\ell+1)\cdot\lrn(H)$.

Every graph of treewidth at most $t$ has simple treewidth at most $t+1$. Therefore, the upper bound in \cref{simple-t-trees} implies the (upper bound in the) following generalization of \cref{trees}:

\begin{thm}\label{t-trees}
    For any fixed integers $\ell\ge 2$, $t\ge 0$, every $n$-vertex graph $H$ of treewidth at most $t$ has $\lrn(H) \in O(\log n/\log^{(t+1)} n)$ and this is asymptotically optimal: for any fixed integer $t\ge 0$ and infinitely many values of $n$, there exists an $n$-vertex graph $H$ of treewidth $t$ with $\trn(H)\in\Omega(\log n/\log^{(t+1)} n)$.
\end{thm}

The lower bound in \cref{t-trees} is through a construction of a treewidth-$t$ graph $H$ with $\trn(H)\in\Omega(\log n/\log^{(t+1)} n)$.  Again, since any graph of treewidth at most $t-1$ has simple treewidth at most $t$, the lower bound in \cref{t-trees} implies the lower bound in \cref{simple-t-trees}.

In addition to planar graphs, there are product structure theorems for a number of other graph classes, including bounded genus graphs, apex minor-free graphs, and $k$-planar graphs.  Using product structure theorems for these graph classes along with \cref{simple-t-trees,t-trees}, we obtain the following two results:

\begin{thm}\label{bounded-genus}
    For any fixed integer $\ell\ge 2$ and any integer $g\ge 0$, every $n$-vertex graph $G$ of Euler genus at most $g$ has $\lrn(G)\in O(g\log n/\log^{(3)} n)$.
\end{thm}

\begin{thm}\label{meta-theorem}\label{meta}
    For each of the following graph classes $\mathcal{G}$:
    \begin{compactenum}
        % \item the class of graphs that have non-crossing drawings in a surface of genus at most $g$;
        \item the class of graphs excluding a particular apex graph $A$ as a minor; and
        \item the class of graphs that can be drawn in a surface of genus $g$ with at most $k$ crossings per edge,
    \end{compactenum}
    there exists an integer $c=c(\mathcal{G})$ such that, for any fixed integer $\ell\ge 2$, every $n$-vertex graph $G\in\mathcal{G}$ has $\lrn(G)\in O(\log n/\log^{(c)} n)$.
\end{thm}

\subsection{Related Work and Relation to Other Colouring Numbers}

Here we survey previous work on $\ell$-ranking as well as its relations to other graph colouring numbers.

\subsubsection{Vertex Ranking}

For a graph $G$, an $\infty$-ranking is known as a \defin{vertex ranking} \cite{bodlaender.deogun.ea:rankings} or \defin{ordered colouring} of $G$ \cite{katchalski.mccuaig.ea:ordered}.  For any graph $G$, $\irn(G)$ is equal to the \defin{treedepth} $\td(G)$ of $G$, defined by \citet{nesetril.ossona:tree-depth} and which plays a central role in the theory of sparsity \cite{nesetril.ossona:on,nesetril.ossona:sparsity}.  Both of these notions are equal to the minimum clique number of a trivially perfect supergraph of $G$ \cite{nesetril.ossona:tree-depth}.

Finding a vertex ranking $\varphi$ that uses exactly $\irn(G)$ colours is equivalent to finding a minimum-height elimination tree of $G$ \cite{torre.greenlaw.ea:optimal,deogun.kloks.ea:on}.  This measure has applications to parallel Cholesky factorization of matrices \cite{bodlaender.gilbert.ea:approximating,duff.reid:multifrontal,liu:role,dereniowski.kubale:cholesky} and in VLSI layout \cite{leiserson:area,sen.deng.ea:on}.  More recently, \citet{even.smorodinsky:hitting} showed that $\irn(G)$ determines the competitive ratio of the best algorithm for the online hitting set problem in $G$.

The \defin{vertex ranking problem} of determining $\irn(G)$ for an arbitrary graph $G$ is known to be NP-hard, even on some restricted classes of graphs \cite{bodlaender.deogun.ea:rankings,llewellyn.tovey.ea:local,llewellyn.tovey.ea:erratum,dereniowski.nadolski:vertex}. Polynomial-time algorithms for the vertex ranking problem have been found for several families of graphs: \citet{schaeffer:optimal,iyer.ratliff.ea:optimal} showed this for trees and \citet{deogun.kloks.ea:on} showed this for permutation graphs.

A straightforward application of divide-and-conquer using planar separators shows that, for any $n$-vertex planar graph $G$, $\irn(G) \in O(\sqrt{n})$ \cite{llewellyn.tovey.ea:local,katchalski.mccuaig.ea:ordered}, and this bound is optimal:  For the $\sqrt{n}\times\sqrt{n}$ grid, $\irn(G)\in \Omega(\sqrt{n})$ \cite{katchalski.mccuaig.ea:ordered}.  A lower bound of \citet{katchalski.mccuaig.ea:ordered} shows that upper bounds like this, using divide-and-conquer with separators, are essentially tight: If, for every $r$-element set $S\subseteq V(G)$, the graph $G-S$ has a component of size at least $\alpha n$, then $\irn(G) \in\Omega(\alpha r)$. In a similar vein, \citet{bodlaender.gilbert.ea:approximating,kloks:treewidth} show that $\irn(G)$ is lower bounded by 1 plus the pathwidth of $G$.

It is not hard to see that, even for an $n$-vertex path $P$, $\irn(P)\in\Omega(\log n)$ and, in fact $\irn(P)=\ceil{\log_2(n+1)}$ \cite{nesetril.ossona:tree-depth}.  The same separator argument, applied carefully to treewidth-$t$ graphs shows that every $n$-vertex treewidth-$t$ graph $G$ has $\irn(G)\le (t+1)\log_2 n$ \cite{nesetril.ossona:tree-depth}.  This shows that, even for graphs with constant-size separators, (worst-case asymptotically) optimal bounds are obtained by divide-and-conquer using separators.  More references on vertex ranking are available in Section 7.19 of the dynamic survey by \citet{gallian:dynamic}.

\subsubsection{$2$-Ranking}

At least three works have considered $\lrn$ for finite $\ell$ with a focus on the case $\ell=2$.  These results are summarized in \cref{summary-table}.

\begin{table}
    \centering{
        \begin{tabular}{|l|c|c|l|} \hline
            Graph class & Upper Bound & Lower Bound & Ref. \\ \hline
            Trees & $O(\log n/\log\log n)$ & $\Omega(\log n/\log\log n)$ & \cite{karpas.neiman.ea:on} \\
            Planar graphs & $O(\ell\log n)$ & $\Omega(\log n/\log\log n)$ & \cite{karpas.neiman.ea:on} \\
            Proper minor closed & $O(\ell\log n)$ & $\Omega(\log n/\log\log n)$ & \cite{karpas.neiman.ea:on} \\
            $d$-cubes & $d+1$ & $d+1$ & \cite{almeter.demircan.ea:graph} \\
            Max-degree 3 & $7$ & & \cite{almeter.demircan.ea:graph} \\
            Max-degree $\Delta$ & $O(\min\{\Delta^2,\Delta\sqrt{n}\})$ & $\Omega(\Delta^2/\log \Delta)$ & \cite{karpas.neiman.ea:on,almeter.demircan.ea:graph} \\
            $d$-degenerate & $O(d\sqrt{n})$ & $\Omega(n^{1/3} + d^2/\log d)$ & \cite{karpas.neiman.ea:on,almeter.demircan.ea:graph} \\
            \hline \multicolumn{4}{c}{} \\
            \hline
            Simple treewidth $\le t$ & $O(\log n/\log^{(t)} n)$ & $\Omega(\log n/\log^{(t)} n)$ & \cref{simple-t-trees} \\
            Treewidth $\le t$ & $O(\log n/\log^{(t+1)} n)$ & $\Omega(\log n/\log^{(t+1)} n)$ & \cref{t-trees} \\
            Planar graphs & $O(\log n/\log^{(3)} n)$ & $\Omega(\log n/\log^{(3)} n)$ & \cref{planar,t-trees} \\
            Outerplanar graphs & $O(\log n/\log^{(2)} n)$ & $\Omega(\log n/\log^{(2)} n)$ & \cref{t-trees}, \cite{karpas.neiman.ea:on} \\
            Genus-$g$ graphs & $O(g\log n/\log^{(3)} n)$ & $\Omega(\log n/\log^{(3)} n)$ & \cref{bounded-genus,t-trees} \\
            $A$-minor-free (apex $A$) & $O(\log n/\log^{(c(A))} n)$ & $\Uparrow$ & \cref{meta} \\
            $(g,k)$-planar & $O(\log n/\log^{(c(g,k))} n)$ & $\Uparrow$  & \cref{meta} \\
            \hline
        \end{tabular}
    } % centering
    \caption{Summary of previous and new results on $\trn$.  All new upper bounds hold for any constant $\ell$. An up-arrow ($\Uparrow$) indicates a lower bound that is implied by the lower bound in the cell directly above. All new lower bounds hold for $\ell=2$. Prior upper bounds hold only for $\ell=2$, with the exception of the $O(\ell\log n)$ upper bound for planar graphs.}
\label{summary-table}
\end{table}

\citet{karpas.neiman.ea:on} proved \cref{trees}---a tight bound of $\trn(T)\in O(\log n/\log\log n)$ for every $n$-vertex tree $T$---and \cref{planar-graphs}---the upper bound $\lrn(G)\in O(\ell\log n)$ for every $n$-vertex planar graph $G$ and every integer $\ell\ge 2$.  More generally, the same authors show that, for any fixed proper minor-closed family $\mathcal{G}$ of graphs $\lrn(G)\in O(\ell\log n)$ for every positive integer $\ell$ and every $n$-vertex $G\in\mathcal{G}$.  They also show that, for fixed $d$, every $n$-vertex $d$-degenerate graph $G$ has $\trn(G)\in O(\sqrt{n})$ and there exists examples with $\trn(G)\in\Omega(n^{1/3})$.

\citet{shalu.antony:complexity} show that determining the minimum number of colours required by a 2-ranking of a given graph is NP-hard, even when restricted to planar bipartite graphs.  \citet{almeter.demircan.ea:graph} determine the exact value of $\trn(Q_d)=d+1$ where $Q_d$ is the $d$-cube.  They also show that, for graphs $G$ of maximum degree 3, $\trn(G)\le 7$ and show the existence of a graph with maximum degree $k$ such that $\trn(G)\in\Omega(k^2/\log k)$.

\subsubsection{Star Colouring and Distance-$2$ colouring}

 Note that 2-rankings fall between two very well-studied graph colouring problems:
\begin{compactitem}
    \item \defin{star colourings}, which ensure that the graph induced by any 2 colour classes is a forest of stars and
    \item \defin{distance-2 colourings} which ensure that the endpoints of each non-trivial path of length at most 2 receive distinct colours.
\end{compactitem}
Every $2$-ranking is a star colouring and every distance-2 colouring is a 2-ranking so, letting $\scn(G)$ and $\dtcn(G)$ denote the star colouring number of $G$ and distance-2 colouring number of $G$, respectively, we have $\scn(G) \le \trn(G)\le \dtcn(G)$.

\subsubsection{Centered Colouring}

% It is not hard to see that any colouring $\varphi$ of $G$ is an $\ell$-ranking if and only if, for every (connected) subgraph $X\subseteq G$ of diameter at most $\ell$, there there is exactly one vertex $v\in V(X)$ such that $\varphi(v)=\max\{\varphi(v):v\in V(X)\}$.

% A \defin{conflict-free} colouring $\varphi:V(G)\to\N$ has the property that, for each vertex $v$ of $G$, some vertex $w$ in the closed neighbourhood $N_G[v]$ has a colour $\varphi(w)$ distinct from every other vertex in $N_G[v]$ \cite{pach.tardos:conflict-free,glebov.szabo.ea:conflict-free,gargano.rescigno:complexity}. A proper colouring $\varphi$ is a $2$-ranking if and only if, for each $v\in V(G)$, $N_G[v]$ has a unique maximum colour.  Therefore every $2$-ranking is a conflict-free colouring.

A colouring $\varphi:V(G)\to\N$ is \defin{$p$-centered} if each connected subgraph $X\subseteq G$ that is coloured with at most $p$ distinct colours has a colour that occurs exactly once \cite{nesetril.ossona:tree-depth,nesetril.ossona:grad,zhu:colouring}.  This implies that in a $p$-centered colouring, every connected subgraph with at most $2p$ vertices must have a colour that occurs exactly once.  In particular, every path of length at most $2p-1$ must have a colour that occurs exactly once.  On the other hand \cref{equivalence} in \cref{basics} shows that $\varphi$ is a $(2p-1)$-ranking if and only if every path of length at most $(2p-1)$ has a unique maximum colour.

This example shows how the difference between ``unique'' and ``unique maximum'' can be surprisingly profound.
% Every planar graph has a conflict-free 4-colouring \cite{abel.alvarez.ea:conflict-free} and the Four Colour Theorem implies that every planar graph has a proper conflict-free $16$-colouring.
Planar graphs (and, indeed, all graph families having similar product structure theorems) have $2p-1$-centered colourings using a number of colours that depends only (polynomially) on $p$ \cite{nesetril.ossona:grad,debski.felsner.ea:improved,pilipczuk.siebertz:polynomial}.  This contrasts starkly with the lower bounds in \cref{trees,planar}, which show that $\ell$-rankings of $n$-vertex trees and planar graphs, respectively, require a number of colours that depends nearly-logarithmically on $n$, for any $\ell\ge 2$.

\subsection{Outline}

The remainder of this paper is organized as follows: \Cref{basics} reviews some basic tools used in the following sections.   \Cref{lower-bounds} proves the lower bound in \cref{t-trees}, which immediately implies the lower bounds in \cref{planar,simple-t-trees}. \Cref{upper-bounds} proves the upper bound in \cref{simple-t-trees}, from which the upper bounds in \cref{planar,t-trees,bounded-genus,meta} follow easily.  \Cref{conclusion} gives a brief summary and discusses directions for further work.

\section{Preliminaries}
\label{basics}

In this paper we use standard graph theory terminology as used in the book by \citet{diestel:graph}.  Every graph $G$ we consider is finite, simple, and undirected with vertex set denoted by $V(G)$ and edge set denoted by $E(G)$.  We use the shorthand $|G|:=|V(G)|$ to denote the number of vertices in $G$.  We use $N_G(v):=\{w\in V(G): vw\in E(G)\}$ to denote the \defin{open neighbourhood} of $v$ in $G$. For any $S\subseteq V(G)$, $N_G(S):=\bigcup_{v\in S} N_G(v)\setminus S$.  For each $n\in\N$, $K_n$ denotes the complete graph on $n$ vertices.  The \defin{length} of a path $u_0,\ldots,u_p$ in $G$ is equal to the number, $p$, of edges in the path. A path is \defin{trivial} if it has length 0 and \defin{non-trivial} otherwise.

For any set $S$, $G[S]$ is the graph with vertex set $V(G[S]):=V(G)\cap S$ and edge set $E(G[S]):=\{vw\in E(G): \{v,w\}\subseteq S\}$, and $G-S:=G[V(G)\setminus S]$.  We say that a subgraph $G'$ of $G$ is an \defin{induced} subgraph of $G$ if $G[V(G')]=G'$.  Although $\ell$-ranking is defined in terms of subgraphs of diameter at most $\ell$, it is more convenient to use an equivalent definition based on (induced) paths:\footnote{Definition (c) is, in fact, the definition used by \citet{karpas.neiman.ea:on}. We only provide a proof of equivalence here for the sake of completeness.}

\begin{obs}\label{equivalence}
  For any graph $G$, any vertex colouring $\varphi:V(G)\to\N$, and any $\ell\in\N\setminus\{0\}$ the following statements are equivalent:
  \begin{compactenum}[(a)]
    \item $\varphi$ is an $\ell$-ranking of $G$.
    \item For every non-trivial path $u_0,\ldots,u_p$ in $G$ of length at most $\ell$, there is exactly one $i\in\{0,\ldots,p\}$ such that $\varphi(u_i)=\max\{\varphi(u_j):j\in\{0,\ldots,p\}\}$.
    \item For every non-trivial path $u_0,\ldots,u_p$ in $G$ of length at most $\ell$,
    \begin{inparaenum}[(i)]
      \item $\varphi(u_0)\neq\varphi(u_p)$; or \item $\varphi(u_0)<\max\{\varphi(u_0),\ldots,\varphi(u_p)\}\}$.
    \end{inparaenum}
  \end{compactenum}
\end{obs}

\begin{proof}
  That $(a)\Rightarrow (b)$ follows immediately from the fact that every path in $G$ of length at most $\ell$ is a connected subgraph of $G$ of diameter at most $\ell$.  To see that $(b)\Rightarrow(c)$ observe that, if $\varphi(u_0)=\varphi(u_p)$ then $\varphi(u_0)\neq\max\{\varphi(u_j):j\in\{0,\ldots,p\}\}$, so $\varphi(u_0)<\max\{\varphi(u_j):j\in\{0,\ldots,p\}\}$.

  To see that $(c)\Rightarrow (a)$ we prove the contrapositive $\neg(a)\Rightarrow\neg(c)$. Suppose that $\varphi$ is not an $\ell$-ranking of $G$. Then $G$ contains a subgraph $X$ with $\diam(X)\le\ell$ that has two vertices $v,w\in V(X)$ such that $\varphi(v)=\varphi(w)=\max\{\varphi(u):u\in V(X)\}$.  Then let $u_0,\ldots,u_p$ be a shortest path in $X$ from $u_0:=v$ to $u_p:=w$.  This path has length $p\le\diam(X)\le\ell$, $\varphi(u_0)=\varphi(u_p)$ and $\varphi(u_0)\not<\max\{\varphi(u_j):j\in\{0,\ldots,p\}\}$.
\end{proof}

\begin{obs}\label{induced-paths-only}
    A colouring $\varphi:V(G)\to\N$ of a graph $G$ is an $\ell$-ranking of $G$ if and only if, for every induced path $u_0,\ldots,u_p$ in $G$ of length at most $\ell$,
    \begin{inparaenum}[(i)]
        \item $\varphi(u_0)\neq\varphi(u_p)$; or
        \item $\varphi(u_0)<\max\{\varphi(u_0),\ldots,\varphi(u_p)\}$.
    \end{inparaenum}
\end{obs}

\begin{proof}
    By \cref{equivalence}(c) any $\ell$-ranking $\varphi$ of $G$ satisfies (i) or (ii) for every path of length at most $\ell$, including every induced path of length at most $\ell$, so this direction is trivial.

    For the other direction, suppose $G$ contains a (not necessarily induced) path $u_0,\ldots,u_p$ of length $p\le\ell$ with $\varphi(u_0)=\varphi(u_p)$ and $\varphi(u_0)=\max\{\varphi(u_0),\ldots,\varphi(u_p)\}$.  Let $w_0,\ldots,w_s$ be the shortest path from $w_0:=u_0$ to $w_s:=u_p$ in  the graph $G[\{u_0,\ldots,u_p\}]$.  Then $w_0,\ldots,w_s$ is an induced path in $G$ with $\varphi(w_0)=\varphi(u_0)=\varphi(u_p)=\varphi(w_s)$ and, since $\{w_0,\ldots,w_{s}\}\subseteq\{u_0,\ldots,u_{r}\}$, $\max\{\varphi(w_0),\ldots,\varphi(w_{s})\}\le\max\{u_0,\ldots,u_{r}\}$, so $\varphi(w_0)=\varphi(u_0)=\max\{\varphi(w_0),\ldots,\varphi(w_s)\}$, as required.
\end{proof}

From this point on, we will use the characterization in \cref{induced-paths-only} as our definition of $\ell$-ranking.  Specifically, in order to prove that some colouring $\varphi:V(G)\to\N$ is an $\ell$-ranking of $G$ we need only show that for any induced path $u_0,\ldots,u_p$, $\varphi(u_0)=\varphi(u_p)$ and $p\le\ell$ implies that $\varphi(u_0)<\max\{\varphi(u_i):i\in\{0,\ldots,p\}\}$.

% The same reasoning used to prove \cref{induced-paths-only} also shows:
%
% \begin{obs}\label{walks-too}
%     A colouring $\varphi:V(G)\to\N$ of a graph $G$ is an $\ell$-ranking of $G$ if and only if, for every walk $w_0,\ldots,w_q$ in $G$ with $w_0\neq w_q$ of length at most $\ell$,
%     \begin{inparaenum}[(i)]
%         \item $\varphi(w_0)\neq\varphi(w_q)$; or
%         \item $\varphi(w_0)<\max\{\varphi(w_0),\ldots,\varphi(w_q)\}$.
%     \end{inparaenum}
% \end{obs}
%

Let $T$ be a tree rooted at some node $r\in V(T)$.  For any node $x\in V(T)$, $P_T(x)$ denotes the path, in $T$, from $r$ to $x$.  The \defin{$T$-depth} of $x\in V(T)$, denoted by $d_T(x)$, is the length of $P_T(x)$.  The \defin{height} of $T$ is $\max\{d_T(x):x\in v(T)\}$.  A node $a\in V(T)$ is a \defin{$T$-ancestor} of $x\in V(T)$ if $a\in V(P_T(x))$. If $a$ is a $T$-ancestor of $x$ then $x$ is a \defin{$T$-descendant} of $a$.  Note that every node of $T$ is both a $T$-ancestor and $T$-descendant of itself.  If $a$ is a $T$-ancestor of $x$ and $x\neq a$ then $a$ is a \defin{strict} $T$-ancestor of $x$ and $x$ is a \defin{strict} $T$-descendant of $a$.  The strict ancestor relation induces a partial order $\prec_T$ on $V(T)$ in which $x\prec_T y$ if and only if $x$ is a strict $T$-ancestor of $y$.

For any graph $G$, and any two vertices $v,w\in V(G)$, $d_G(v,w)$ denotes the length of a shortest path, in $G$, from $v$ to $w$ or $d_G(v,w):=\infty$ if $v$ and $w$ are in different connected components of $G$. The \defin{diameter} of $G$ is $\diam(G):=\max\{d_G(v,w):v,w\in V(G)\}$. For any integer $k\ge 1$, the \defin{$k$-th power} of $G$, denoted by $G^k$, is the graph with vertex set $V(G^k):=V(G)$ and edge set $E(G^{k}):=\{vw:v,w\in V(G),\, 1\le d_G(v,w)\le k\}$.
% In the special case $k=2$, $G^2$ is called the \defin{square} of $G$.
Note that any distance-$\ell$ colouring of $G$ is a proper colouring of $G^\ell$ and vice-versa, i.e., $\dlcn(G)=\chi(G^\ell)$.

For any $v\in V(G)$ and any $W\subseteq V(G)$, let $d_G(v,W)=\min\{d_G(v,w):w\in W\}$. A \defin{(generalized) BFS layering} of a connected graph $G$ is a partition of $V(G)$ into a sequence $\mathcal{L}:=(L_0,\ldots,L_m)$ of sets such that, for each $i\in\{1,\ldots,m\}$ and each $v\in L_i$, $d_G(v,L_0)=i$.  Any BFS layering $\mathcal{L}:=(L_0,\ldots,L_m)$ defines a partial order $\prec_{\mathcal{L}}$ on $V(G)$ in which $v\prec_{\mathcal{L}} w$ if and only if $v\in L_i$, $w\in L_j$ and $i<j$.

% A path $P:=u_0,\ldots,u_p$ in $G$ is an \defin{induced} path if $G[\{u_0,\ldots,u_p\}]$ is a path.

\subsection{Graph Decompositions, Treewidth, and Pathwidth}

For two graphs $H$ and $X$, an \defin{$X$-decomposition} of $H$ is a sequence $\mathcal{X}:=(B_x:x\in V(X))$ of subsets of $V(H)$ called \defin{bags} indexed by the nodes of $X$ and such that
 \begin{inparaenum}[(i)]
     \item for each $v\in V(H)$, $X[\{x\in V(X):v\in B_x\}]$ is connected; and
     \item for each $vw\in E(H)$, there exists some $x\in V(X)$ such that $\{v,w\}\subseteq B_x$.
\end{inparaenum}
The \defin{width} of $\mathcal{X}$ is $\max\{|B_x|:x\in V(X)\}-1$. We say that $H$ is \defin{edge-maximal} with respect to $\mathcal{X}$ if, for each $x\in V(X)$, the vertices in $B_x$ form a clique in $H$.

% The $X$-decomposition $\mathcal{X}$ is \defin{smooth} if, for each edge $xy\in E(X)$, $|B_x\setminus B_y|\le 1$ and $|B_y\setminus B_x|\le 1$.

In the special case where $X$ is a tree (or a forest), $\mathcal{X}$ is called a \defin{tree decomposition} of $H$.  In the more special case where $X$ is a path (or a collection of disjoint paths), $\mathcal{X}$ is called a \defin{path decomposition} of $H$. The \defin{treewidth} $\tw(H)$ of $H$ is the minimum width of any tree decomposition of $H$. The \defin{pathwidth} $\pw(H)$ of $H$ is the minimum width of any path decomposition of $H$.  If a graph $G$ is edge-maximal with respect to a path decomposition (tree decomposition) $\mathcal{X}$ of width $t$, then $G$ is an interval graph (chordal graph, respectively) whose maximum clique size is $t+1$ \cite{bodlaender:partial}.

For a graph $H$, a \defin{rooted tree decomposition} of $H$ is a tree decomposition $\mathcal{T}:=(B_x:x\in V(T))$ of $H$ in which $T$ is a rooted tree.  Throughout the remainder of the paper, all our tree decompositions are rooted, with the root of $T$ typically denoted by $r$, in which case we call it an \defin{$r$-rooted tree decomposition}.  We use the notation $x_\mathcal{T}(v)$ to denote the minimum $T$-depth node $x\in V(T)$ such that $v\in B_x$.  This induces a partial order $\prec_{\mathcal{T}}$ on $V(H)$ in which $v\prec_{\mathcal{T}} w$ if and only if $x_\mathcal{T}(v)\prec_T x_\mathcal{T}(w)$.  The following observations have straightforward proofs:

 \begin{obs}\label{induced-unimodal}
     Let $H$ be a graph that is edge-maximal with respect to some rooted tree decomposition $\mathcal{T}:=(B_x:x\in V(T))$ of $H$.  Then, for any induced path $u_0,\ldots,u_p$ in $H$ and any $i\in\{1,\ldots,p-1\}$, $u_i\preceq_\mathcal{T} u_0$ or $u_i\preceq_\mathcal{T} u_p$.
 \end{obs}

\begin{obs}\label{order-relation}
    Let $H$ be a connected graph that is edge-maximal with respect to an $r$-rooted tree decomposition $\mathcal{T}:=(B_x:x\in V(T))$ of $H$ and let $\mathcal{L}:=L_0,\ldots,L_m$ be a BFS layering of $H$ with $L_0:=B_r$.  Then, for any $v,w\in V(H)$, $v\prec_{\mathcal{T}}w$ implies $v\preceq_{\mathcal{L}}w$. Equivalently, there is no pair $v,w\in V(H)$ such that $v\prec_{\mathcal{L}}w$ and $w\prec_{\mathcal{T}}v$.
\end{obs}

\begin{obs}\label{up-neighbours}
    Let $H$ be a connected graph that is edge-maximal with respect to a width-$t$ $r$-rooted tree decomposition $\mathcal{T}:=(B_x:x\in V(T))$ of $H$ and let $\mathcal{L}:=L_0,\ldots,L_m$ be a BFS layering of $H$ with $L_0:=B_r$. Then, for any $i\in\{1,\ldots,m\}$ and any component $X$ of $H[\bigcup_{j=i}^m L_j]$, $L_{i-1}\cap N_H(V(X))$ is contained in a single bag $B_x$ of $\mathcal{T}$.
\end{obs}

% \begin{proof}
%     By \cref{order-relation}, $w\preceq_\mathcal{T} v$ for each $w\in N_H(v)\cap L_i$.  Proprety~(i) of tree decompositions therefore implies that $N_H(v)\cap L_{i-1}\subseteq B_{x_\mathcal{T}(v)}$.  Therefore $t+1\ge |B_{x_\mathcal{T}(v)}| \ge |N_H(v)\cap L_{i-1}|+1$, as required.
% \end{proof}

We will make use of the following fairly standard vertex-weighted separator lemma.  Similar lemmas with similar proofs appear in  \citet{robertson.seymour:graph}, but we provide a proof for the sake of completeness.

\begin{lem}\label{weighted-separator}
    Let $H$ be a graph; let $\mathcal{T}:=(B_x:x\in V(T))$ be a tree decomposition of $H$; and let $\xi:V(H)\to\R$ be a function that is positive on $V(H)$.  Then, for any $c\in\N\setminus\{0\}$, there exists $S_T\subseteq V(T)$ of size $|S_T|\le c-1$ such that, for each component $X$ of $H-(\bigcup_{x\in S_T} B_x)$, $\sum_{v\in V(X)} \xi(v) \le \tfrac{1}{c}\cdot\sum_{v\in V(H)} \xi(v)$.
\end{lem}

\begin{proof}
  Let $\Xi:=\sum_{v\in V(H)} \xi(v)$.
  The proof is by induction $c$.  The base case $c=1$ is trivial, since $S_T:=\emptyset$ satisfies the requirements of the lemma.  Now assume $c\ge 2$.  Root $T$ at some arbitrary vertex $r$ and for each $x\in V(T)$, let $T_x$ denote the subtree of $T$ induced by $x$ and all its $T$-descendants.  Let $G_x:=G[\bigcup_{y\in V(T_x)} B_y]$.  Say that a node $x$ of $T$ is \defin{heavy} if $\sum_{v\in V(G_x)} \xi(v) \ge \tfrac{1}{c}\cdot \Xi$. Since $c\ge 1$, $r$ is heavy, so $T$ contains at least one heavy vertex. Let $x$ be a heavy vertex of $T$ with the property that no child of $x$ is also heavy.  Then $G':=G-V(G_x)$ has weight $\sum_{v\in V(G')} \xi(v) \le (1-1/c)\cdot\Xi$.  On the other hand, every component $C$ of $G-V(G')-B_x$ has weight $\sum_{v\in V(C)} \xi(v) \le \tfrac{1}{c}\cdot\Xi$.  Apply induction on the graph $G'$ with tree decomposition $\mathcal{T}':=(B_x\cap V(G'):x\in V(T))$ and $c':=c-1$ to obtain a set $S_T'$ of size at most $c-2$ such that each component $X$ of $G'-(\bigcup_{x\in S_T} B_x)$, has weight at most $\sum_{v\in V(X)} \xi(v) \le \tfrac{1}{c-1}\cdot(1-\tfrac{1}{c})\cdot\Xi = \tfrac{1}{c}\cdot \Xi$.  The set $S_T:=S_T'\cup\{x\}$ satisfies the requirements of the lemma.
\end{proof}

% \begin{proof}
%     (This proof is only included for the sake of being self-contained.)
%     Define $n_0:=\tfrac{1}{c}\cdot\sum_{v\in V(H)} \xi(v)$ and, for a subtree $T'$ of $T$, let $n_{T'}:=\sum_{y\in \bigcup_{x\in V(T')}B_x]} B_y$.  Let $T^{(0)}:=T$, let $i:=0$ and repeat the following steps until $n_{T^{(i)}} < n_0$:
%     \begin{compactenum}
%         \item Find a node $x_{i}\in V(T^{(i)})$ of maximal $T^{(i)}$-depth such that the subtree $T^{(i)}_{x_i}$ of $T^{(i)}$ containing $x_i$ and its $T^{(i)}$-descendants has $n_{T^{(i)}_{x_i}} > n_0$.
%
%         \item Set $T^{(i+1)}$ to be the component of $T^i-\{x_i\}$ containing $r$ and set $i:=i+1$.
%     \end{compactenum}
%     The key observation is that, since $x_i$ is of maximal depth, each component $X$ of $T^{(i)}-\{x_i\}$ other than $T^{(i+1)}$ has $n_X\le n_0$.
%     This procedure certainly terminates after at most $n_T/n_0 = O(c)$ steps and the set $S_T:=\{x_0,\ldots,x_p\}$ has the required properties.
% \end{proof}

\subsection{Simple Treewidth}

A tree decomposition $\mathcal{T}:=(B_x:x\in V(T))$ of a graph $H$ is \defin{$t$-simple} if it has width at most $t$ and, for each $t$-element subset $S\subseteq V(H)$, $|\{x\in V(T):S\subseteq B_x\}|\le 2$.  The \defin{simple treewidth} $\stw(H)$ of a graph $H$ is the minimum integer $t$ such that $H$ has a $t$-simple tree decomposition \cite{knauer.ueckerdt:simple}.  \citet{knauer.ueckerdt:simple} define simple treewidth and  \citet{wulf:stacked} studies it extensively in his thesis.
% \citet{markenzon.justel.ea:subclasses} define simple $t$-trees (which they call Simple Clique (SC) $t$-trees).

% Simple treewidth is minor-monotone. This is stated by \citet{knauer.ueckerdt:simple} and \citet[Theorem~5.2]{wulf:stacked} gives a proof:
%
% \begin{lem}[\cite{knauer.ueckerdt:simple,wulf:stacked}]\label{simple-minor-closed}
%     For every graph $G$ and every minor $M$ of $G$, $\stw(M)\le\stw(G)$.
% \end{lem}

We work with simple treewidth because it arises naturally in the graphs we are interested in:

\begin{lem}[\cite{knauer.ueckerdt:simple,markenzon.justel.ea:subclasses}]\label{simple-small-cases}
    For any graph $H$,
    \begin{compactenum}[(i)]
        \item $\stw(H)\le 1$ if and only if $H$ is a collection of vertex-disjoint paths;
        \item $\stw(H)\le 2$ if and only if $H$ is outerplanar;
        \item $\stw(H)\le 3$ if and only if $\tw(H)\le 3$ and $H$ is planar.
    \end{compactenum}
\end{lem}

Simple treewidth and treewidth are closely related:

\begin{lem}[\cite{knauer.ueckerdt:simple}]\label{simple-treewidth-vs-treewidth}
    For every graph $G$, $\tw(G)\le \stw(G)\le \tw(G)+1$.
\end{lem}

% The following two results show that, although not immediately obvious, simple treewidth and simple $t$-trees behave very much like treewidth and $t$-trees.  The first result, due to \citet[Theorem~3.27]{wulf:stacked}, shows that every graph of simple treewidth at most $t$ is a spanning subgraph of some simple $t$-tree.\footnote{The statement of \cite[Theorem~3.27]{wulf:stacked} does not explicitly state that $V(G)=V(H)$, but it is the case.  The relevant location is the proof that (iii)$\Rightarrow$(i) in Theorem~3.25 where, by definition, $V(\mathrm{fill}(H))=V(G)$.}
%
% [TODO: We can get away without this, so let's clean it up when we get the chance.]
% \begin{lem}[\cite{wulf:stacked}]\label{simple-subgraph}
%     For any graph $G$ with $\stw(G)\le t$, there exists a simple $t$-tree $H$ with $V(G)= V(H)$ and $E(G)\subseteq E(H)$.
% \end{lem}

The following lemma, whose proof uses minor-monotonicity \cite[Theorem~5.2]{wulf:stacked}, is due to David\ R.\ Wood (personal communication).

\begin{lem}\label{simple-bfs-layers}
    Let $H$ be a connected graph that is edge-maximal with respect to some $r$-rooted $t$-simple tree decomposition $\mathcal{T}:=(B_x:x\in V(T))$ of $H$ and let $L_0,\ldots,L_m$ be the BFS ordering of $H$ with $L_0:=B_r$.   Then, for each $i\in\{1,\ldots,m\}$, $\stw(H[L_i])\le t-1$.
\end{lem}

\subsection{Product Structure}

For two graphs $G_1$ and $G_2$, the \defin{strong graph product} of $G_1$ and $G_2$, denoted $G_1\boxtimes G_2$, is a graph whose vertex set is the Cartesian product $V(G_1)\times V(G_2)$ and that contains an edge between $v=(v_1,v_2)$ and $w=(w_1,w_2)$ if and only if
\begin{inparaenum}[(i)]
    \item $v_1=w_1$ and $v_2w_2\in E(G_2)$;
    \item $v_2=w_2$ and $v_1w_1\in E(G_1)$; or
    \item $v_1w_1\in E(G_1)$ and $v_2w_2\in E(G_2)$.
\end{inparaenum}

The following result of \citet{dujmovic.joret.ea:planar}, which builds on earlier work of \citet{pilipczuk.siebertz:polynomial}, shows that every planar graph is the subgraph of a strong product of very simple graphs.

\begin{thm}[\cite{dujmovic.joret.ea:planar}]\label{product-structure}
    For every $n$-vertex planar graph $G$, there exists a graph $H$, $|H|\le n$, $\stw(H)\le 3$, and a path $P$ such that $G$ is isomorphic to a subgraph of $H\boxtimes K_3\boxtimes P$.
\end{thm}

As the following simple lemma shows, product structure is highly relevant to $\ell$-ranking:

\begin{lem}\label{product-lemma}
    For any two graphs $G_1$ and $G_2$, $\lrn(G_1\boxtimes G_2)\le \lrn(G_1)\cdot\dlcn(G_2)$.
\end{lem}

\begin{proof}
    For each $(x,y)\in V(G_1\boxtimes G_2)$, let $\varphi(x,y):=\dlcn(G_2)\cdot \rho(x) - \psi(y)$ where $\rho:V(G_1)\to\{1,\ldots,\lrn(G_1)\}$ is an $\ell$-ranking of $G_1$ and $\psi:V(G_2)\to\{0,\ldots,\dlcn(G_2)-1\}$ is a distance-$\ell$ colouring of $G_2$.

    To see that $\varphi$ is an $\ell$-ranking, consider any
    path $u_0,\ldots,u_p$ in $G_1\boxtimes G_2$ of length $p\le\ell$ such that $\varphi(u_0)=\varphi(u_p)$.  We must show that $\varphi(u_0)<\max\{\varphi(u_0),\ldots,\varphi(u_p)\}$.

    For each $i\in\{0,\ldots,p\}$, let $(u_{i,1},u_{i,2}):=u_i$, so that $u_{i,1}\in V(G_1)$ and $u_{i,2}\in V(G_2)$. Since $\varphi(u_0)=\varphi(u_p)$, $\psi(u_{0,2})=\psi(u_{p,2})$. Since $\psi$ is a distance-$\ell$ colouring of $G_2$ and $p\le\ell$, this implies that $u_{0,2}=u_{p,2}$.  This implies that $u_{0,1}\neq u_{p,1}$, for otherwise $u_0=u_p$ and $u_0,\ldots,u_p$ is not a path.  Therefore, $u_{0,1},\ldots,u_{p,1}$ is a walk in $G_1$ with distinct endpoints.  Let $w_0,\ldots,w_q$ be a shortest path from $w_0:=u_{0,1}$ to $w_q:=u_{p,1}$ in $G_1[\{u_{0,1},\ldots,u_{p,1}\}]$.

    Since $\rho(u_{0,1})=\rho(u_{p,1})$, $\rho$ is an $\ell$-ranking of $G_1$, and $q\le p\le\ell$,   $\rho(u_{0,1})=\rho(w_{0})<\max\{\rho(w_0),\ldots,\rho(w_q)\}\le\max\{\rho(u_{0,1}),\ldots,\rho(u_{p,1})\}$ and therefore $\varphi(u_0)<\max\{\varphi(u_0),\ldots,\varphi(u_p)\}$, as required.
\end{proof}

Note that the graph $K_3\boxtimes P$, which appears in \cref{product-structure}, has maximum degree 8 so $(K_3\boxtimes P)^\ell$ has maximum degree at most $8\cdot 7^{\ell-1}$.  Since distance-$\ell$ colouring any graph $G$ is equivalent to properly colouring $G^{\ell}$, this implies that $\dlcn(K_3\boxtimes P)\le 8\cdot7^\ell+1$. The following observation improves this constant using the fact that $(K_d\boxtimes P)^\ell$ is $(d(\ell+1)-1)$-degenerate (as can be seen by ordering vertices of $(K_d\boxtimes P)$ by the order that their second coordinate appears in $P$).

\begin{obs}\label{dumb}
    For any $d\in\N$ and any path $P$, $\dlcn(K_d\boxtimes P)\le d(\ell+1)$.
\end{obs}

% \begin{proof}
%     Order the vertices of $K_3\boxtimes P$ as $(x_1,y_1),\ldots,(x_{3|P|},y_{3|P|})$ so that, if $y_i$ appears before $y_j$ in $P$, then $i<j$. For each $j\in\{1,\ldots,3|P|\}$, let $I_j\subset\{1,\ldots,j-1\}$ be the set of indices $i$ such that $K_3\boxtimes P$ contains a non-trivial path of length at most $\ell$ with endpoints $(x_i,y_i)$ and $(x_j,y_j)$.  It is easy to verify that $|I_j|\le 3(\ell+1)$.  For each $j\in\{1,\ldots,3|P|\}$, set $\varphi(x_j,y_j):=\min(\{1,\ldots,3(\ell+1)\}\setminus\{\varphi(x_i,y_i):i\in I_j\})$.  This gives a distance-$\ell$ colouring $\varphi:K_3\boxtimes P\to\{1,\ldots,3(\ell+1)\}$.
% \end{proof}
%

We remark that \cref{dumb} is tight since, for any path $P$ of length at least $\ell$, $(K_d\boxtimes P)^\ell$ contains cliques of order $d(\ell+1)$.

\subsection{Inequalities for Iterated Logarithms}

For any $x> 0$ and $a\ge 0$, we have the inequality,
\begin{equation}
    \log (x+a) = \log (x(1+a/x)) = \log x + \log(1+a/x)
    % \le \log x + \log e^{a/x}
    \le \log x + \frac{a}{x} \enspace , \label{log-x-plus-a}
\end{equation}
where the inequality follows from the inequality $1+z\le e^z$, valid for all $z\in\R$.

Recall that, for any integer $i\ge 0$,
\[
    \log^{(i)} x :=
      \begin{cases}
          x & \text{for $i=0$} \\
          \log\left(\log^{(i-1)}x\right) & \text{for $i\ge 1$.} \\
      \end{cases}
\]
Define the \defin{$\tau$ower} function $\tau:\N\to\N$ by
\[
  \tau(i) :=
    \begin{cases}
        1 & \text{for $i=0$} \\
        e^{\tau(i-1)} & \text{for $i\ge 1$.} \\
    \end{cases}
\]
Note that, for all $i\in\N$, $\log^{(i)}\tau(i)=1$.

For any $x > \tau(i-1)$ and any $a\ge 0$, \cref{log-x-plus-a} generalizes as follows (by induction on $i$):
\begin{equation}
    \log^{(i)}(x+a) \le \log^{(i)} x + \frac{a}{\prod_{j=0}^{i-1}\log^{(j)} x} \label{logi-x-plus-a}
\end{equation}

In several places we have ratios involving iterated logarithms, in which case we make use of the following consequence of \cref{logi-x-plus-a}
\begin{equation}
    \frac{\log^{(i)} (x+a)}{\log^{(i)} x} \le 1 + \frac{a}{\prod_{j=0}^{i}\log^{(j)} x} \enspace, \label{logi-ratio}
\end{equation}
which is valid for all $x> \tau(i-1)$.

\subsection{The $\gamma_{i,k}$ Function}

For any $i\in\N\setminus\{0\}$, any real $k>\tau(i-1)$, and any real $n\in[1,(\log^{(i)} k)^k]$, we define $\gamma_{i,k}(n)$ to be the solution $x\in[\tau(i),k]$ to the equation
\begin{equation}
   (\log^{(i)} k)^k/(\log^{(i)} x)^{x}=n \enspace .  \label{gamma_eqn}
\end{equation}
The value of $\gamma_{i,k}(n)$ is well defined and $\tau(i)\le \gamma_{i,k}(n)\le k$, for the following reasons:  For $x\in[\tau(i),k]$, the left hand side of \cref{gamma_eqn} is a continuous strictly decreasing function of $x$. Setting $x=\tau(i)$, the left hand side becomes $(\log^{(i)} k)^k \ge n$. Setting $x=k$, the left hand side becomes $1\le n$.  We note (and later make use of) the fact that $\gamma_{i,k}(n)$ is a decreasing function of $n\in [\tau(i),k]$.

%========================================================================
\section{Lower Bounds}
\label{lower-bounds}

We now prove the lower bound in \cref{t-trees}, which establishes all the other lower bounds. The idea is to construct a graph $G$ that has a BFS layering $L_0,\ldots,L_m$ such that, for each $i\in\{0,\ldots,m-1\}$ and each vertex $a\in L_i$, $G[N_G(a)\cap L_{i+1}]$ is a collection of treewidth-$(t-1)$ graphs $U_{a,0},\ldots,U_{a,k}$, each of which is a copy of a small treewidth-$(t-1)$ graph $U$ that requires at least $h$ colours.  This forces the colour of $a$ to exceed, by at least $h$, the smallest colour used in $U_{a,0},\ldots,U_{a,k}$.  Proceeding bottom up, this forces the vertex in $L_0$ to receive a colour larger than $hm$.  The lower bound is then obtained by using induction on $t$ to upper bound the size of the graph $U$ needed to ensure that $\trn(U)\ge h$ and choosing the parameters $h$ and $m$ appropriately.

\begin{lem}\label{apex-graph}
    Let $h,k\in\N\setminus\{0\}$, let $U$ be a graph with $\trn(U)\ge h$, and let $G$ be a graph obtained by taking $k+1$ disjoint copies $U_0,\ldots,U_k$ of $U$ and adding an apex vertex $a$ adjacent to each $v\in\bigcup_{i=0}^k V(U_i)$.  Let $\varphi:V(G)\to\{1,\ldots,k\}$ be a $2$-ranking of $G$ with the property that $\varphi(v)\ge k_0$ for each $v\in\bigcup_{i=0}^k V(U_i)$ and some $k_0\in\{1,\ldots,k-h\}$.  Then $\varphi(a) \ge k_0+h$.
\end{lem}

\begin{proof}
    Since $\trn(U_i)\ge h$ and each $v\in V(U_i)$ has $\varphi(v)\ge k_0$, there exists some $v_i\in V(U_i)$ such that $\varphi(v_i)\ge k_0+h-1$, for each $i\in\{0,\ldots,k\}$.  Since $|\{0,\ldots,k\}|=k+1>k-k_0+1=|\{k_0,\ldots,k\}|$ the Pigeonhole Principle implies that there exists distinct $i,j\in\{0,\ldots,k\}$ such that $\varphi(v_i)=\varphi(v_j)$.  Since $v_i a v_j$ is a path in $G$, this implies that $\varphi(a)\ge \varphi(v_i)+1\ge k_0+h$.
\end{proof}

For a graph $U$ and integers $h,m\ge 0$, we define the \defin{$(h,m)$-boost} $U^{(h,m)}$ of $U$ as follows: The vertex set of $U^{(h,m)}$ is the disjoint union of $L_0,\ldots,L_m$.  The set $L_0:=\{a_0\}$ consists of a single vertex. For each $i\in\{1,\ldots,m\}$ and each $a\in L_{i-1}$, $U^{(h,m)}$ contains $hm+1$ disjoint copies $U_{a,0},\ldots,U_{a,hm}$ of $U$ and contains the edge $av$ for each $v\in\bigcup_{j=0}^{hm} V(U_{a,j})$.  This determines the set $L_i=\bigcup_{a\in L_{i-1}}\bigcup_{j=0}^{hm} V(U_{a,j})$.  As a simple example, if $U$ is a 1-vertex graph, then $U^{(h,m)}$ is a complete $(hm+1)$-ary tree of height $m$.

\begin{lem}\label{boost}
    For any non-empty graph $U$, any $m\in\N$, and any integer $h\in\{1,\ldots,\trn(U)\}$, $\trn(U^{(h,m)})\ge hm +1$.
\end{lem}

\begin{proof}
    Let $k:=\trn(U^{(h,m)})$ and let $\varphi:V(U^{(h,m)})\to\{1,\ldots,k\}$ be a 2-ranking of $U^{(h,m)}$.  Let $L_0,\ldots,L_{m}$ be the partition of $V(U^{(h,m)})$ used in the definition of $U^{(h,m)}$. We will show by induction on $m-i$ that, for each $a\in L_{i}$, $\varphi(a)\ge(m-i)h+1$. Since $a_0\in L_0$, this gives $k\ge \varphi(a_0)\ge m h+1$.

    The base case of the induction, $m-i=0$, is trivial; it simply asserts that $\varphi(v)\ge 1$ for each $v\in L_m$.  Now assume $m-i > 0$.  For each $a\in L_i$ and each  $v\in N_{U^{(h,m)}}(a)\cap L_{i+1}$, the inductive hypothesis implies that $\varphi(v)\in\{(m-i-1)h+1,\ldots,k\}$.  The induced graph $G:=U^{(h,m)}[\{a\}\cup N_{U^{(h,m)}}(a)\cap L_{i+1}]$ is the graph described by \cref{apex-graph} with the value $k_0:=(m-i-1)h+1$.  The conclusion of \cref{apex-graph} therefore implies that $\varphi(a)\ge k_0+h=(m-i)h+1$, as required.
\end{proof}

\begin{lem}\label{boost-size}
    For any graph $U$ and any $h,m\in\N\setminus\{0\}$, $|U^{(h,m)}| \le (|U|hm)^{m}\cdot e^{O(1/h)}$.
\end{lem}

\begin{proof}
    It is easy to see that, for each $i\in \{0,\ldots,m\}$, $|L_i|=(|U|(hm+1))^i$.  Therefore,
    \begin{align*}
      |U^{h,m}|
        & = \sum_{i=0}^m |L_i| \\
        & = \sum_{i=0}^m (|U|(hm+1))^i \\
        % & = \sum_{i=0}^m (|U|hme^{1/hm})^i \\
        % & = (|U|hm)^m \cdot e^{1/h} \\
        %
        %
        & = (|U|(hm+1))^{m}\cdot (1+O(1/(|U|hm))) \\
        & \le (|U|hme^{1/hm})^{m}\cdot e^{O(1/(|U|hm))} \\
        & = (|U|hm)^{m}\cdot e^{1/h+O(1/(|U|hm))} = (|U|hm)^{m}\cdot e^{O(1/h)}
        \enspace . \qedhere
    \end{align*}
\end{proof}

\begin{lem}\label{boost-treewidth}
    For any graph $U$ and any integers $h,m\ge 1$, $\tw(U^{(h,m)})\le \tw(U)+1$.
\end{lem}

\begin{proof}
  Let $t:=\tw(U)$.
  Create a width-$(t+1)$ tree-decomposition $(B_x:x\in V(T))$ of $U^{(h,m)}$ as follows: Start with $T$ having a single node $z_0$ with $B_{z_0}=L_0$.  For each $i\in\{1,\ldots,m\}$, and each $a\in L_{i-1}$, find some bag $B_z$ in the current decomposition that contains $a$, take $hm+1$ disjoint copies $(A_x:x\in V(T_0)),\ldots,(A_x:x\in V(T_h))$ of some width-$t$ tree decomposition $\mathcal{T}$ of $U$.  For each $i\in\{0,\ldots,hm\}$, add an edge from $z$ to any node of the tree in $T_i$ and add $a$ to every bag in $T_i$.  It is straightforward to verify that this does, indeed, give a width-$(\tw(U)+1)$ tree-decomposition of $U^{(h,m)}$.
\end{proof}

\begin{lem}\label{treewidth-lower-bound}
    For each $t\in\N\setminus\{0\}$ and every integer $r\ge \tau(t)$, there exists a graph $G$, with $|G|\le (\log^{(t-1)}r)^{r + o(r)}\cdot e^{(t-1)r+o(r)}$, $\tw(G)\le t$, and $\trn(G)\ge r$.
\end{lem}

\begin{proof}
    The proof is by induction on $t$.  \citet{karpas.neiman.ea:on} have shown that the complete $(r+1)$-ary tree $T$ of height $r-1$ has $\trn(T)\ge r$. As the tree $T$ has size $\sum_{i=0}^{r-1} (r+1)^i \le r^r=(\log^{(0)}r)^{r}\cdot e^{0}$,  this establishes the base case $t=1$.

    Let $h:=\lceil\log r\rceil$ and $m:=\lceil r/\log r\rceil$ so that $hm\ge r$.  For $t>1$ we can apply the inductive hypothesis to obtain a graph $U$, with $\tw(U)\le t-1$, $|U|\le (\log^{(t-2)} h)^{(t-1)h+o(h)}$ and $\trn(U)\ge h$. Let $G:=U^{(h,m)}$.  By \cref{boost-treewidth}, $\tw(G)\le \tw(U)+1\le t$.  By \cref{boost}, $\trn(G)\ge hm+1 > hm \ge r$. By \cref{boost-size},
    \begin{align*}
        |G| & \le (|U|\cdot m\cdot h)^{m}\cdot e^{O(1/h)} \\
        & \le \left((\log^{(t-2)} h)^{h + o(h)}\cdot e^{(t-2)h+o(h)} \cdot mh\right)^{m}\cdot e^{O(1/h)} \\
        & = (\log^{(t-2)} h)^{r + o(r)}\cdot e^{(t-2)r+o(r)}\cdot e^{r+o(r)}\cdot e^{O(1/h)}
        & \text{(since $h=\lceil\log r\rceil$ and $m=\lceil r/\log r\rceil$)}\\
        & = (\log^{(t-1)} r)^{r + o(r)}\cdot e^{(t-1)r+o(r)}\enspace . & & \qedhere
    \end{align*}
\end{proof}

\begin{proof}[Proof of \cref{t-trees} (lower bound)]
    By \cref{treewidth-lower-bound} there exists an $n$-vertex graph $G$ with $n\le (\log^{(t-1)} r)^{r+o(r)}\cdot e^{(t-1)r+o(r)}$, $\tw(G)\le t$, and $\trn(G)\ge r$.  So, for any fixed $t\in\N$,
    \[
        \log n \le (r+o(r))\log^{(t)} r
        \le (1+o(1))\cdot \trn(G)\cdot \log^{(t)} \trn(G) \enspace .
    \]
    and attempting to solve for $\trn(G)$ shows that $\trn(G)\in \Omega(\log n/\log^{(t+1)} n)$.
\end{proof}

The lower bound construction in this section gives some guidance on how to obtain a matching upper bound for $\trn(G)$.  Specifically, for some node $a\in L_i$, the colouring of the component $X$ of $H[\{a\}\cup\bigcup_{j=i+1}^m L_j]$ that contains $a$ can create a lower bound on $\varphi(a)$.  Specifically, if two vertices $u,w\in V(X[L_{i+1}])$ receives the same colour $\phi$ then $\varphi(a)>\phi$.  This suggests that one should attempt to minimize the largest colour that is repeated in the colouring of $X[L_{i+1}]$.  Indeed, this is a guiding principle in our upper bound proof.

%========================================================================
\section{Upper Bounds}
\label{upper-bounds}

In this section we prove asymptotically tight bounds for the worst-case number of colours needed for $\ell$-ranking simple treewidth-$t$ graphs, treewidth-$t$ graphs, planar graphs, and bounded genus graphs. In order to avoid complicating an already technically demanding proof, for the rest of this section we will treat $\ell$ and $t$ as fixed constants independent of $n$ and other parameters that are unbounded, so that $f(\ell,t)\in O(1)$ for any function $f:\N\times\N\to\N$.  At the end of this section, in \cref{dependence-on-ell} we discuss the dependence of $\lrn$ on $\ell$.

\subsection{Simple Treewidth-$t$ Graphs}
\label{simple-treewidth-section}

This section is devoted to proving the upper bound in \cref{simple-t-trees}:

\begin{namedtheorem}[\weirdref{simple-t-trees}{a}]\weirdlabel{simple-t-trees}{a}
    For fixed integers $\ell\ge 2$, $t\ge 1$, every $n$-vertex graph $H$ with $\stw(H)\le t$ has $\lrn(H)\in O(\log n/\log^{(t)} n)$.
\end{namedtheorem}

\weirdref{simple-t-trees}{a} immediately implies the upper bounds in \cref{planar,t-trees}:

\begin{proof}[Proof of \cref{planar} (upper bound)]
    By \cref{product-structure}, $G$ is a subgraph of $H\boxtimes K_3\boxtimes P$ where $|H|\le n$, $\stw(H)\le 3$, and $P$ is a path. Therefore,
    \begin{align*}
        \lrn(G) & \le \lrn(H\boxtimes K_3\boxtimes P)
                    & \text{(by \cref{product-structure})}\\
                & \le \lrn(H)\cdot \dlcn(K_3\boxtimes P)
                    & \text{(by \cref{product-lemma})} \\
                & \le 3(\ell+1)\cdot\lrn(H) & \text{(by \cref{dumb})} \\
                & \in O(\log n/\log^{(3)} n) & \text{(by \weirdref{simple-t-trees}{a}).} & \qedhere
    \end{align*}
\end{proof}

\begin{proof}[Proof of \cref{t-trees} (upper bound)]
    By \cref{simple-treewidth-vs-treewidth}, $\stw(H)\le\tw(H)+1\le t+1$ so, by   \weirdref{simple-t-trees}{a}, $\lrn(H)\in O(\log n/\log^{(t+1)}n)$.
\end{proof}

\weirdref{simple-t-trees}{a} also has the following corollary, which strengthens \cref{trees}:

\begin{cor}\label{outerplanar}
    For each fixed integer $\ell\ge 2$, every $n$-vertex outerplanar graph $G$ has $\lrn(G)\in O(\log n/\log^{(2)} n)$.
\end{cor}

\begin{proof}
    By \cref{simple-small-cases}{(ii)}, $\stw(G)\le 2$ so, by \weirdref{simple-t-trees}{a} $\lrn(G)\in O(\log n/\log^{(2)} n)$.
\end{proof}

The proof of \weirdref{simple-t-trees}{a} is the most technically demanding part of the paper and is the subject of most of this section.  Globally, the proof is by induction on the value of $t$, though it is easy to miss this, since it is spread over several lemmas. The case $t=1$ is easy: By \cref{simple-small-cases}(i), any graph of simple treewidth 1 is a contained in a path and therefore has an $\ell$-ranking using $\ell+1\in O(\log n/\log^{(1)} n) = O(1)$ colours.\footnote{\label{logarithmic-base-case}This is far from tight: Any path has an $\ell$-ranking using at most $k:=\floor{\log_2 \ell}+2$ colours.  Any path of length at most $2^{k-1}-1$ is easily coloured using colours in $\{1,\ldots,k-1\}$ using divide-and-conquer \cite{nesetril.ossona:tree-depth}.  To colour a path $v_0,\ldots,v_m$ with $m\ge 2^{k-1}$, set $\varphi(v_i):=k$ for each $i\equiv 0\pmod{2^{k-1}}$.  Then the set of uncoloured vertices induces a collection of paths each of length at most $2^{k-1}-1$ which can be coloured using colours in $\{1,\ldots,k-1\}$.}  In the proof of \cref{t-tree-slack}, below, we will apply \weirdref{simple-t-trees}{a} to graphs of simple treewidth $t-1$. \cref{t-tree-slack} is then used in the proof of \cref{t-tree-technical} which is used in the proof of \weirdref{simple-t-trees}{a} (a statement about graphs of simple treewidth $t$), at the end of this section.

\subsubsection{The Bread}

We begin with a few helper lemmas whose purpose is to show that, for a graph $H$ having a width-$t$ tree-decomposition $\mathcal{T}:=(B_x:x\in V(T))$, $\lrn(H)$ can be bounded by a function of $t$ and the number of branching (degree at least $3$) nodes in $T$. We begin with the simplest case: when $H$ has a width-$t$ path decomposition.

\begin{lem}\label{pathwidth}
    For any graph $G$, $\lrn(G)\le (\ell+1)\pw(G) + 1$.
\end{lem}

\begin{proof}
    The proof is by induction on $\pw(G)$.  The base case $\pw(G)=0$ is trivial: In this case, $G$ contains no edges and can be $\ell$-ranked with $1 = (\ell+1)\pw(G)+1$ colours.   For $\pw(G)\ge 1$, we may assume that $G$ is connected since, otherwise, we can colour each component of $G$ separately. Let $P:=x_1,\ldots,x_m$ be a path and let $(B_x:x\in V(P))$ be a $P$-decomposition of $G$ of width $\pw(G)$.

    Let $v_1,\ldots,v_p$ be a path of minimum length such that $v_1\in B_{x_0}$ and $v_p\in B_{x_m}$.  Since $v_1,\ldots,v_m$ is a path in $G$ with $v_1\in B_{x_0}$ and $v_r\in B_{x_m}$,  $|B_x\cap\{v_1,\ldots,v_p\}|\ge 1$, for each $x\in V(P)$.  Since $(B_x\setminus\{v_1,\ldots,v_p\}:x\in V(P))$ is a path decomposition of $G-\{v_1,\ldots,v_p\}$, this implies that $\pw(G-\{v_1,\ldots,v_p\})\le \pw(G)-1$.  We inductively colour $G-\{v_1,\ldots,v_p\}$ using colours $\{1,\ldots,(\ell+1)(\pw(G)-1)+1\}$ and then colour each $v_i$ with colour $((\ell+1)(\pw(G)-1)+2+i)\bmod (\ell+1)$.

    A standard property of shortest paths implies that, for each $i,j\in\{1,\ldots,p\}$, $|j-i|=d_G(v_i,v_j)$. In this colouring, $\varphi(v_i)=\varphi(v_j)$ implies that $j-i\equiv 0\pmod{\ell+1}$, for any $i,j\in\{1,\ldots,p\}$.  In particular, for distinct $i,j\in\{1,\ldots,p\}$, $\varphi(v_i)=\varphi(v_j)$ implies that $d_G(v_i,v_j) = |j-i|\ge\ell+1$.

    To see that the resulting colouring is an $\ell$-ranking, consider any path $X$ in $G$ of length at most $\ell$. If $V(X)\cap\{v_1,\ldots,v_p\}\neq\emptyset$ then each vertex $V(X)\cap\{v_1,\ldots,v_m\}$ has a unique colour, which is larger than any colour used by any vertex in $V(X)\setminus\{v_1,\ldots,v_m\}$.  If $V(X)\cap\{v_1,\ldots,v_p\}=\emptyset$ then $X\subseteq G-\{v_1,\ldots,v_p\}$ has a unique maximum colour by the inductive hypothesis.
\end{proof}

\begin{lem}\label{path-induced}
    Let $P=x_1,\ldots,x_m$ be a path and let $G$ be a graph that is edge-maximal with respect to a width-$t$ $P$-decomposition $\mathcal{P}:=(B_x:x\in V(P))$ of $G$.  Then there exists a set $U\subseteq V(G)$ such that
    \begin{compactenum}[(Z1)]
        \item $B_{x_1}\cup B_{x_m}\subseteq U$;\label{values-u}
        \item $|U|\le 2(\ell+1)^t + t$; and \label{size-u}
        \item for each non-trivial induced path $w_0,\ldots,w_q$ in $G$ of length at most $\ell$, $\{w_0,w_q\}\subseteq U$ implies that $\{w_1,\ldots,w_{q-1}\}\subseteq U$.\label{induced-u}
    \end{compactenum}
\end{lem}

\begin{proof}
    To eliminate a level of subscripts, let $x_i:=i$ for each $i\in\{1,\ldots,m\}$. The proof is by induction on $t$. In the base case, $t=0$, $G$ has no edges and therefore no non-trivial paths, so Z\ref{induced-u} is vacuous. The lemma is satisfied by taking $U:=B_{1}\cup B_{m}$.  This certainly satisfies Z\ref{values-u} and satisfies Z\ref{size-u} since $|U|\le 2= 2(\ell+1)^0+0$.

    Now assume that $t\ge 1$. If $G$ is not connected, then $B_{1}$ and $B_{m}$ are in different components of $G$.  In this case we choose $U:=B_{1}\cup B_{m}$. This certainly satisfies Z\ref{values-u}.  This satisfies Z\ref{size-u} since $|U|\le|B_1|+|B_m|\le 2t+2 \le 2(\ell+1)^t+t$ because $2(\ell+1)^t \ge 2^{t+1}> t+2$ for all $t\ge 1$.  This also satisfies Z\ref{induced-u} because the only paths $w_0,\ldots,w_q$ that need consideration have $\{w_0,w_q\}\subseteq B_1$ or $\{w_0,w_q\}\subseteq B_m$.  Since we only consider induced paths in $G$ and $G$ is edge-maximal with respect to $\mathcal{P}$, this implies that $q=1$, so $w_0,\ldots,w_q=w_0w_q$ consists of a single edge and $\{w_0,w_q\}\subseteq U$.

    We may now assume that $G$ is connected. For each $v\in V(G)$, let $r(v):=\max\{i\in\{1,\ldots,m\}:v\in B_i\}$.  Let $y_0:=1$ and $i:=1$.  As long as $y_i\neq m$, choose a vertex $u_i\in B_{y_i}$ that maximizes $y_{i+1}:=r(u_i)$ and increment $i$. This produces a path $u_0,\ldots,u_p$ in $G$ and a sequence of nodes $y_0,\ldots,y_{p+1}$ in $V(P)$.  It is easy to verify that $u_0,\ldots,u_p$ is a shortest path from $B_{1}$ to $B_{m}$, i.e., $p=\min\{d_H(w_0,w_q): w_0\in B_1,\, w_q\in B_m\}$.  Therefore, if $p>\ell$, the lemma is again trivially satisfied by taking $U:=B_{1}\cup B_{m}$.
    %
    %
    %
    % Consider the \defin{greedy path} $u_0,\ldots,u_p$ that begins at $u_0\in B_{1}$, ends at $u_p\in B_{m}$, and is defined as follows: $u_0$ is a vertex in $B_{1}$ that maximizes $r(u_0)$.  For $i\ge 1$, $u_i$ is a vertex in $B_{r(u_{i-1})}$ that maximizes $r(u_i)$.  This continues until reaching a vertex $u_p$ such that $r(u_p)=m$.
    %
    %
    %
    % \begin{figure}
    %   \begin{center}
    %     \includegraphics{figs/path_induced-1} \\[0pt]
    %     \includegraphics{figs/path_induced-2} \\
    %   \end{center}
    %   \caption{The proof of \cref{path-induced}.}
    %   \label{induced_path_fig}
    % \end{figure}

    % Otherwise, $u_0,\ldots,u_p$ defines a sequence $y_0,\ldots,y_{p+1}$ of nodes in $P$, where $y_0:=1$ and, for each $i\in\{1,\ldots,p+1\}$, $y_{i}:=r(u_{i-1})$.
    Now assume that $p\le\ell$.  For each $i\in\{1,\ldots,p+1\}$, define the path $P_i=y_{i-1},\ldots,y_i$, let $\mathcal{P}_i:=(B_x\setminus\{u_{i-1}\}:x\in V(P_i))$, and let $G_i:=G[B_x\setminus\{u_{i-1}\}:x\in V(P_i)]$.  Then $G_i$ is edge-maximal with respect to $\mathcal{P}_i$ and $\mathcal{P}_i$ has width at most $t-1$. For each $i\in\{1,\ldots,p\}$, we apply the lemma inductively to $G_i$ and $\mathcal{P}_i$ to obtain a set $U_i$.  Let $U:=\{u_0,\ldots,u_p\}\cup\bigcup_{i=1}^{p+1} U_i$.  Observe that, by induction, $\bigcup_{i=1}^{p+1}U_i\supseteq\bigcup_{i=1}^{p+1}(B_{y_{i-1}}\cup B_{y_i}\setminus\{u_{i-1}\})$, so $U\supseteq\bigcup_{i=0}^{p+1} B_{y_i}$.

    % For each $i\in\{1,\ldots,p\}$, $U$ contains $u_{i-1}$ and $u_i$. The set $U_i$ is obtained by induction on $G_i$ using a path decomposition on a path whose endpoints are $y_{i-1}$ and $y_i$ having bags $B_{y_{i-1}}\setminus u_{i-1}$ and $B_{y_{i}}\setminus u_{i-1}$.  Therefore $B_{y_{i-1}}\cup B_{y_i}\subseteq U_i\subseteq U$.  Therefore $\bigcup_{j=0}^{p+1} B_{y_j}\subseteq U$.

    In particular, $U$ contains $B_{y_0}=B_1$ and $B_{y_{p+1}}=B_m$, so $U$ satisfies (Z\ref{values-u}).
    Now observe that $|U|\le |B_1| + |\{u_1,\ldots,u_p\}|+ |\bigcup_{i=1}^{p+1}|U_i\setminus B_{y_{i-1}}|$.  Since $p\le \ell$, this implies that $|U\setminus B_1|$ satisfies the recurrence
    \[
       |U\setminus B_1|\le f(t) \le \begin{cases}
                  1 & \text{if $t=0$} \\
                  \ell + (\ell+1)\cdot f(t-1) & \text{otherwise.}
                \end{cases}
    \]
    This recurrence resolves to $f(t)\le 2(\ell+1)^t - 1$.  Therefore $|U|\le f(t)+|B_{1}| \le 2(\ell+1)^t + t$ so this satisfies Z\ref{size-u}.  All that remains is to show that $U$ satisfies (Z\ref{induced-u}).  Consider some induced path $w_0,\ldots,w_q$ in $G$ of length at most $\ell$ with $\{w_0,w_q\}\subseteq U$.  We want to show that $\{w_1,\ldots,w_{q-1}\}\subseteq U$.

    We say that a vertex $w_i$ is \defin{pinched} if $w_i\in B_{y_j}$ for some $j\in\{0,\ldots,p+1\}$. (Note that each of $u_0,\ldots,u_p$ is pinched.) The edges of $w_0,\ldots,w_q$ can be partitioned into subpaths of the form $w_{a},\ldots,w_{b}$ where
    \begin{inparaenum}[(i)]
        \item $w_a$ is pinched;
        \item $w_b$ is pinched; and
        \item none of $w_{a+1},\ldots,w_{b-1}$ are pinched.
    \end{inparaenum}
    First note that, for any such subpath $w_a,\ldots,w_b$, $\{w_a,w_b\}\subseteq U$, so we need only show that $\{w_{a+1},\ldots,w_{b-1}\}\subseteq U$.  There are three cases to consider:

    \begin{compactenum}
       \item $\{w_a,w_b\}\subseteq B_{y_j}$ for some $j\in\{0,\ldots,p+1\}$.  Since $G$ is edge-maximal with respect to $\mathcal{P}$, this implies  that $w_aw_b\in E(G)$. Since $w_a,\ldots,w_b$ is an induced path in $G$, $b=a+1$ and there is nothing to prove.

       \item $\{w_a,w_b\}\subseteq V(G_j)$ for some $j\in\{1,\ldots,p\}$ (and not the preceding case).
       % In this case, edge-maximality implies that $B_{y_{j-1}}$ and $B_{y_{j}}$ each form cliques that separate $G'_{j}$ from $G-V(G'_j)$.
       Since
       % $w_a,\ldots,w_b$ is an induced path in $G$ and
       none of $w_{a+1},\ldots,w_{b-1}$ are pinched, this implies that $\{w_a,\ldots,w_b\}\subseteq V(G_j)$.  Therefore, $w_a,\ldots,w_b$ is an induced path in $G_j$ so, by the inductive hypothesis, $\{w_{a+1},\ldots,w_{b-1}\}\subseteq U_j\subseteq U$.
       % (Note that this includes the special case in which $u_j\in\{w_a,\ldots,w_b\}$.)

       \item $w_a = u_{j-1}$ for some $j\in\{1,\ldots,p+1\}$ and $w_b\in V(G_{j})$.  In this case, $w_a=u_{j-1}\in B_{k}$ for each $k\in\{y_{j-1},\ldots,y_{j}\}$ and $w_b\in B_{k}$ for at least one $k\in\{y_{j-1},\ldots,y_{j}\}$.  By edge maximality, $w_aw_b\in E(G)$, so $b=a+1$ and there is nothing to prove.  \qedhere
       %
       %
       % There are three subcases to consider:
       % \begin{compactenum}
       %      \item $w_{a+1}\in B_{y_{j+1}}$. In this case, $w_{a+1}$ is pinched, so $b=a+1$ and there is nothing to prove.
       %
       %      \item $w_{a+1}=w_q$. In this case $b=q=a+1$ and there is nothing to prove.
       %
       %      \item Neither of the previous two cases. We argue that this is not possible, so the previous two cases are already exhaustive.  Since $w_{a+1}\not\in B_{y_{j+1}}$, $w_{a+1}\in V(G_{j+1}')\setminus B_{y_{j+1}}$. Since $w_{a+1}\neq w_q$, ${a+2}\le q$ and $w_{a+2}\in V(G_{j+1}')$.  However since $G$ is edge-maximal with respect to $\mathcal{P}$, $N_G(u_j)\supseteq V(G_{j+1}')$. In particular, $w_aw_{a+2}\in E(G)$, contradicting the assumption that $w_0,\ldots,w_q$ is an induced path in $G$. \qedhere
       %  \end{compactenum}
    \end{compactenum}
\end{proof}

A node $x$ in a rooted tree $T$ is a \defin{branching node} if $x$ has at least two children.  Let $\Lambda(T)$ denote the set of branching nodes in a tree $T$.  Let $H$ be a graph that is edge-maximal with respect to some tree decomposition $\mathcal{T}:=(B_x:x\in V(T))$ of width at most $t$. We define the \defin{$(\mathcal{T},\ell)$-skeleton} $\hat{H}$ of $H$ as the induced subgraph of $H$ whose vertex set is defined as follows:
\begin{compactenum}
    \item $V(\hat{H})$ contains $\bigcup_{x\in\Lambda(T)} B_x$.

    \item For each pair of nodes $x,y\in\Lambda(T)$ such that the path $P_T(x,y)$ from $x$ to $y$ in $T$ has no branching node in its interior, $V(\hat{H})$ contains the set $U_{xy}\subseteq V(H)$ obtained by applying \cref{path-induced} to the graph $G_{xy}:=H[\bigcup_{z\in V(P_T(x,y))} B_z]$ with the path decomposition $\mathcal{P}_{xy}:=(B_z:z\in P_T(x,y))$.  (Note that $G_{xy}$ and $\mathcal{P}_{xy}$ satisfy the edge-maximality required for \cref{path-induced} since $H$ is edge-maximal with respect to $\mathcal{T}$.)
    % \todo{Can we use reachability in the skeleton to do better here? No!}
\end{compactenum}

\begin{lem}\label{skeleton-paths}
    Let $w_0,\ldots,w_q$ be an induced path in $H$ of length at most $\ell$ and with endpoints $\{w_0,w_q\}\subseteq V(\hat{H})$. Then $\{w_1,\ldots,w_{q-1}\}\subseteq V(\hat{H})$.
\end{lem}

\begin{proof}
    Partition the edges of $w_0,\ldots,w_q$ into paths of the form $w_a,\ldots,w_b$ such that
    \begin{inparaenum}[(i)]
        \item $a=0$ or $w_a\in\bigcup_{x\in \Lambda(T)} B_x$;
        \item $b=q$ or $w_b\in\bigcup_{x\in \Lambda(T)} B_x$; and
        \item none of $w_{a+1},\ldots,w_{b-1}$ are contained $\bigcup_{x\in \Lambda(T)} B_x$.
    \end{inparaenum}
    This means that $w_a,\ldots,w_b$ is an induced path in $G_{xy}$ for some $x,y\in \Lambda(H)$ and $\{w_a,w_b\}\subseteq U_{xy}$.  Therefore, by \cref{path-induced} $\{w_{a+1},\ldots,w_{b-1}\}\subseteq U_{xy}\subseteq U$, as required.
\end{proof}

\begin{lem}\label{skeleton-size}
    $|V(\hat{H})|\le (|\Lambda(T)|-1)\cdot(2(\ell+1)^t+t)$.
\end{lem}

\begin{proof}
    This follows from \cref{path-induced} (Z\ref{size-u}) and the fact that there are $|\Lambda(T)|-1$ distinct pairs $x,y\in\Lambda(T)$ such that $P_T(x,y)$ has no internal nodes in $\Lambda(T)$.
\end{proof}

\begin{lem}\label{skeleton-colour}
    Let $H$ be a graph that is edge-maximal with respect to some width-$t$ tree decomposition $\mathcal{T}:=(B_x:x\in V(T))$ of $H$ that defines a $(\mathcal{T},\ell)$-skeleton $\hat{H}$, of $H$.  Then $\lrn(H)\le \lrn(\hat{H}) + (\ell+1)t+1$.
\end{lem}

\begin{proof}
    Let $\varphi:V(\hat{H})\to \{(\ell+1)t+2,\ldots,\lrn(\hat{H})+(\ell+1)t+1\}$ be an $\ell$-ranking of $\hat{H}$. The graph $P:=T-\Lambda(T)$ consists of disjoint paths and, for any edge $vw\in E(H-V(\hat{H})))$ there is a node $x\in V(P)$ such that $\{v,w\}\subseteq B_x$.  Therefore $(B_x:x\in V(P))$ is a width-$t$ path decomposition of $H-V(\hat{H})$, so $\pw(H-V(\hat{H}))\le t$.  Therefore, by \cref{pathwidth}, $H-V(\hat{H})$ has an $\ell$-ranking $\varphi:V(H-V(\hat{H}))\to\{1,\ldots,(\ell+1)t+1\}$.  This gives a colouring $\varphi: V(H)\to\{1,\ldots,\lrn(\hat{H})+(\ell+1)t+1\}$.

    We claim that $\varphi$ is an $\ell$-ranking of $\hat{H}$.  To see this, consider some induced path $u_0,\ldots,u_p$ in $H$ with $\varphi(u_0)=\varphi(u_p)$.  We must show that $\varphi(u_i)>\varphi(u_0)$ for some $i\in\{1,\ldots,p-1\}$.
    Since $\rho(u_0)=\rho(u_p)$ and the colours used to colour $\hat{H}$ are distinct from those used to colour $H-V(\hat{H})$, there are only two cases to consider:
    \begin{compactenum}
        \item $\{u_0,u_p\}\subseteq V(H-V(\hat{H}))$. There are two subcases:
        \begin{compactenum}
            \item $\{u_1,\ldots,u_{r-1}\}\subseteq V(H-V(\hat{H}))$.  In this case, $u_0,\ldots,u_p$ is a path in $H-V(\hat{H})$, so  $\varphi(u_0)<\varphi(u_i)$ for some $i\in\{1,\ldots,p-1\}$ since \cref{pathwidth} ensures that $\varphi$ is an $\ell$-ranking of $H-V(\hat{H})$.

            \item $u_i\in V(\hat{H})$ for some $i\in\{1,\ldots,p-1\}$. In this case, $\varphi(u_0)\le (\ell+1)t+1 < (\ell+1)t+2 \le \varphi(u_i)$.
        \end{compactenum}
        \item $\{u_0,u_p\}\subseteq V(\hat{H})$.
        % In this case, \cref{induced-unimodal} implies that $u_0,\ldots,u_p$ is $\mathcal{T}$-unimodal, i.e., $u_0,\ldots,u_p$ is the concatenation of two paths, one upward and one downward.
        By \cref{skeleton-paths}  $\{u_0,\ldots,u_p\}\subseteq V(\hat{H})$, so  $\varphi(u_0)<\varphi(u_i)$ for some $i\in\{1,\ldots,p-1\}$ since $\varphi$ is an $\ell$-ranking of $\hat{H}$. \qedhere
    \end{compactenum}
\end{proof}

\subsubsection{The Meat}

Now we arrive at the combinatorial core of the proof. The main idea is to cover $H$ with a sequence of overlapping blocks, each of which consists of $\ell+2$ consecutive BFS layers.  Each pair of consecutive blocks overlaps in a single BFS layer. To convey some intuition about the proof, we first present it for trees.

\paragraph{The Proof for Trees.}

We will now show that, for any $\ell\in\N\setminus\{0\}$, any $k\ge 3$, and any tree $T$ with $n \le k^k$ vertices has a $\ell$-ranking $\varphi:V(T)\to\{0,\ldots,\floor{ak}\}$, for some value of $a$ that depends only on $\ell$.  Observe that a value of $k\in O(\log n/\log\log n)$ is sufficient to satisfy the condition $n\le k^k$, so this already proves that $\lrn(T)\in O(\log n/\log\log n)$, which extends the result of \citet{karpas.neiman.ea:on} that $\trn(T)\in O(\log n/\log\log n)$.

Let $r$ be the root of $T$, let $h$ be the height of $T$ and, for each $i\in\{0,\ldots,h\}$, let $L_i$ denote the set of vertices in $T$ that have depth $i$.  For each vertex $v$ in $T$, let $T_v$ be the subtree of $T$ that contains $v$ and all its descendants, let $n_v:=|T_v|$ be the number of vertices in $T_v$, and let $c_v$ be the solution to the equation $k^k/c_v^{c_v} = n_v$.  In other words, $c_v:=\gamma_{0,k}(n_v)$.  Let $c:=c_r$, so $n=k^k/c^c$.  We will prove the following stronger result (see \cref{tree_fig}):

\begin{figure}
  \begin{center}
    \includegraphics{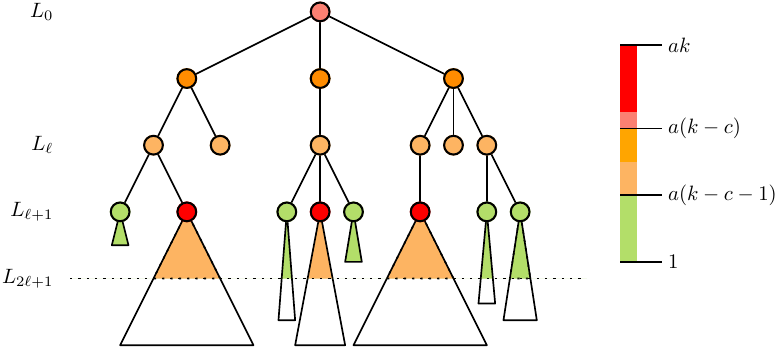}
  \end{center}
  \caption{Finding an $\ell$-ranking of a tree of size $n=k^k/c^c$.  The root gets colour $\floor{a(k-c)}+1$. Vertices in $L_1,\ldots,L_{\ell}$ get colours in $[a(k-c-1),a(k-c)]$.  Dangerous vertices in $L_{\ell+1}$ get unique colours in $[\floor{a(k-c)}+2,ak]$.  Harmless vertices get colours in $[1,a(k-c)]$.}
  \label{tree_fig}
\end{figure}

\begin{quote}
    $T$ has an $\ell$-ranking $\varphi:V(T)\to\{1,\ldots,\floor{ak}\}$ such that $\varphi(r)=\floor{a(k-c)}+1$ and $\varphi(v)\le \floor{a(k-c)}$ for each vertex $v\in\bigcup_{i=1}^\ell L_i$.
\end{quote}

We proceed by induction on $n$. In the base case, $n=1$, so $c=k$.  We set $\varphi(r):=\floor{a(k-c)}+1=1$ and we are done. Now suppose $n\ge 2$.  For each $v\in L_{\ell+1}$, apply the inductive hypothesis on the subtree $T_v$ to obtain an $\ell$-ranking $\varphi$ of the forest $F:=T[\bigcup_{i=\ell+1}^h L_i]$.

We say that a vertex $v\in L_{\ell+1}$ is \defin{dangerous} if $T_v$ has size $n_v> k^k/(c+1)^{c+1}$ and \defin{harmless} otherwise.  Observe that the number $x$ of dangerous vertices must satisfy $xk^k/(c+1)^{c+1} < n = k^k/c^c$, so
\[
  x \le \frac{(c+1)^{c+1}}{c^c} = \left(\frac{c+1}{c}\right)^c\cdot(c+1)
  = (1+1/c)^c\cdot (c+1) < e\cdot (c+1) \enspace .
\]
We modify $\varphi$ by assigning a unique colour $\varphi(v)\in\{\floor{a(k-c)+2,\ldots,\floor{ak}}$ to each dangerous vertex.  The number of colours available for dangerous vertices is at least $ac-2$ and the number of dangerous vertices is at most $e(c+1)$, so this is always possible, provided that
% ac-2 \ge e(c+1)
% ac \ge e(c+1) + 2
% a \ge (e(c+1) + 2)/c
% a \ge e + (e+2)/c
$a \ge 8 > e + (e+2)/c$.  Observe that this modification can only increase the value of $\varphi(v)$, since $c_v \ge c$ and the inductive hypothesis ensures that, prior to this modification $\varphi(v)=\floor{a(k-c_v)}+1 < \floor{a(k-c)}+2$.  This implies that the modified colouring is still a $2$-ranking of $F$ since, by the inductive hypothesis, $\varphi(v)$ is the unique largest colour in $\bigcup_{j=\ell+1}^{2\ell+1}(V(T_v)\cap L_{j})$.

Next observe that each harmless vertex $w\in L_{\ell+1}$ has $c_w\ge c+1$, so $\varphi(w) = \floor{a(k-c_w)}+1 \le \floor{a(k-c-1)}+1 < \floor{a(k-c)}-\ell$ for any $a\ge \ell+2$. Extending $\varphi$ to the vertices in $L_0,\ldots,L_{\ell}$ is now straightforward:  For each $i\in\{0,\ldots,\ell\}$ and each $v\in L_i$, set $\varphi(v):=\floor{a(k-c)}+1-i$.

It is straightforward to check that the resulting colouring $\varphi:V(T)\to\{1,\ldots,\lfloor ak\rfloor\}$ satisfies the stronger conditions of the inductive hypothesis.  To see why $\varphi$ is an $\ell$-ranking of $T$, consider any path $P$ of length at most $\ell$.
\begin{compactitem}
  \item If $P$ is entirely contained in $T_v$ for some $v\in L_{\ell+1}$, then $P$ has a unique maximum colour by the inductive hypothesis.

  \item Otherwise, if $P$ contains no dangerous vertices then the unique maximum colour in $P$ occurs at the unique vertex of $P$ that has minimum $T$-depth.

  \item Otherwise, $P$ contains one or two dangerous vertices that have distinct colours and the largest of these is larger than any other colour that appears in $P$.
\end{compactitem}
This completes the proof for trees.  With some small changes, the proof given above also works for simple $2$-trees, i.e., maximal outerplanar graphs.  The differences are as follows:
\begin{compactitem}
    \item For a simple $2$-tree $H$ we use a BFS layering $L_0,\ldots,L_h$ where $L_0$ may contain a single vertex or both endpoints of an edge of $H$, and the  vertices in $L_0$ will receive colours in $\{\floor{a(k-c)}+1,\floor{a(k-c)}+2\}$.

    \item For each $i\in\{1,\ldots,\ell\}$, the induced graph $H[L_i]$ is a collection of paths, so it is coloured using $\ell+1$ distinct colours in $\{\floor{a(k-c)}+1-i(\ell+1),\ldots,\floor{a(k-c)}+1-(i-1)(\ell+1)\}$.  This works, provided that $a\ge\ell(\ell+1) + 2$.
\end{compactitem}
Note that, for the second point to work, it is crucial that $H$ be a \emph{simple} $2$-tree.  If $H$ is a (not necessarily simple) $2$-tree then $H[L_i]$ can be an arbitrary forest, for which an $\ell$-ranking may require $\log|L_i|/\log\log|L_i|$ colours.

\paragraph{The Proof for Simple $t$-Trees.}

Our proof for simple $t$-trees has some elements in common with the proof presented above:
\begin{compactitem}
  \item It follows the same general outline of first inductively colouring components of $H[\bigcup_{i=\ell+1}^{h} L_i]$ and then increasing the colours of dangerous vertices in layer $L_{\ell+1}$ so that they are all unique.

  \item In the final colouring vertices in $L_1,\ldots,L_\ell$ receive colours not larger than $a(k-c)$, and the vertices in $L_0$ have colours that are larger than all vertices in $L_1,\ldots,L_\ell$.
\end{compactitem}

Unfortunately, this is where the similarities end, and for $t\ge 3$ considerable complications appear that are not present when $t=2$. In short, this happens because, for $t\ge 3$, the graph $H[L_i]$ within each layer may require a number of colours that is not bounded by any function of $\ell$.  We now give a high-level overview of how to deal with this.
% The proof for trees and simple $2$-trees described above creates a colouring in which the colours used in layer $i$ are all greater than those used in layer $i+1$, for each $i\in\{0,\ldots,\ell-1\}$.  This works because for a tree $T$, $T[L_i]$ is a graph without edges, which can be coloured with one colour.  For simple $2$-tree $H$, $H[L_i]$ is a collection of paths, which can be coloured with $\ell+1$ colours.  However, for a simple $t$-tree, the graph $H[L_i]$ is a graph of simple treewidth $t-1$, which can not necessarily be coloured with a number of colours that depends only on $\ell$.  For example, when $H$ is a simple $3$-tree, the graph $H[L_i]$ can contain an arbitrary outerplanar graph on $n'$ vertices, which may require $\Omega(\log n'/\log\log n')$ colours.  Without some control on the size of components in $H[L_i]$, simply  increasing the value of $a$ to accommodate for this will not work, and actually leads to a colouring that uses $\omega(\log n)$ colours.
See \cref{t-tree_fig}

\begin{figure}
  \begin{center}
    \includegraphics{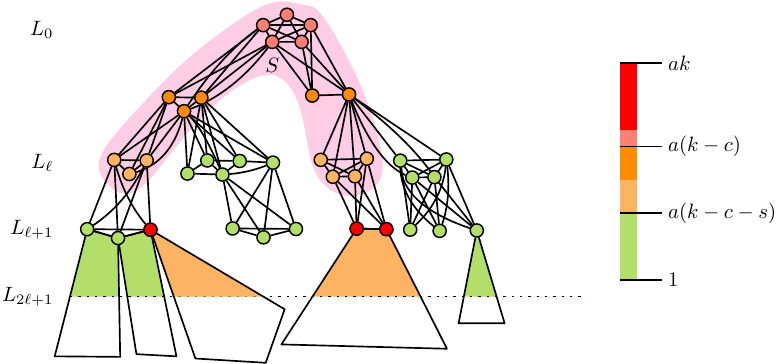}
  \end{center}
  \caption{Finding an $\ell$-ranking of a simple $t$-tree of size $n=(\log^{(t-2)})^k/(\log^{(t-2)}c)^c$. Vertices in $S$ get colours in $[a(k-c-s),a(k-c)]$. Dangerous vertices in $L_{\ell+1}$ get unique colours in $[\floor{a(k-c)},ak]$.  Harmless vertices and vertices in $H[\bigcup_{i=0}^{\ell+1} L_i]-S$ get colours in $[1,a(k-c-s)]$.}
  \label{t-tree_fig}
\end{figure}

We (mostly) give up on the idea of using distinct colours for each of $L_0,\ldots,L_\ell$.  As above, we are inductively colouring a graph $H$ with $n\le (\log^{(t-2)} k)^k$ vertices and $c:=\gamma_{t-2,k}(n)$, and we want to show that $H$ has an $\ell$-ranking that uses at most $\floor{ak}$ colours in which the vertices in $L_1,\ldots,L_{\ell}$ receive colours less than $a(k-c)$.

To achieve this we use a separator $S\subseteq \bigcup_{i=0}^\ell L_i$ that guarantees that the size of each component of $H-S$ is sufficiently small that it has an $\ell$-ranking in which the vertices in $L_0,\ldots,L_\ell$ have colours smaller than $a(k-c-s)$ for some appropriately chosen $s$. For this to work, we need that each component have size at most $(\log^{(t-2)} k)^k/(\log^{(t-2)}(c+s))^{c+s}$.  The situation is complicated further by the fact that the harmless vertices in $L_{\ell+1}$ create lower bounds on the colours of their neighbours in $L_\ell$.  Without these lower bounds, two vertices in distinct components of $H[L_{\ell+1}]$ might receive the same colour which is smaller than the colour of their common neighbour in $L_\ell$.  This requires us to solve a weighted generalization of the problem on $H[\bigcup_{i=0}^{\ell} L_i]$.  This weighted problem on a graph of diameter $d\in O(\ell)$ is the subject of \cref{t-tree-slack}, below.

Assuming this weighted generalization can be solved, this still leaves the problem of colouring the vertices in $S$.  For a carefully chosen value of $s$, the separator $S$ is the union of $O(c^4)$ bags in a $t$-simple tree-decomposition of $H$, where $\log c/\log^{(t-1)}(c)=O(s)$.  We can then augment $S$ into a superset $S'\supseteq S\cup L_0$ in such a way that $H[S']$ has a width-$t$ tree decomposition whose underlying tree has $O(c^4)$ branching nodes.  For each $i\in\{0,\ldots,\ell\}$, the graph $H_i':=H[L_i\cap S']$ has simple treewidth $t-1$ (by \cref{simple-bfs-layers}) and has a tree decomposition of width at most $t-1$ whose underlying tree has $O(c^4)$ branching nodes.  Therefore, by \cref{skeleton-colour} and induction on $t$, $H'_i$ has an $\ell$-ranking using $O(\log c^4/\log^{(t-1)} c^4)=O(s)$ colours. This leaves enough room to colour all of $H'$ using colours in the interval $[\floor{a(k-c-s)}+1,\floor{a(k-c)}]$, using a colouring in which all vertices of $H'_i$ have larger colours than those of $H'_{i+1}$, for each $i\in\{0,\ldots,\ell-1\}$.

In order for all of this to work, we must strike a balance between the size of the separator $S$ and the sizes of the components that remain after removing $S$.  As it turns out, setting $s:=\log c/\log^{(t-1)} c$ achieves what we need.  This choice of $s$ appears in the following lemma, which is what we eventually use to colour the vertices in $L_0,\ldots,L_\ell$.  The purpose of the weighting $(n_v:v\in V(H))$ that appears in this lemma is to deal with the fact, discussed above, that harmless vertices in $L_{\ell+1}$ that are coloured inductively will place lower bounds on the colours of vertices in $L_{\ell}$.

\begin{lem}\label{t-tree-slack}
    Let $t,d,\ell\in\N\setminus\{0\}$ be fixed values, let $k\ge 3$; let $H$ be a graph with $\diam(H)\le d$ and $\stw(H)\le t$ in which each vertex $v\in V(H)$ is assigned a real-valued weight $n_v\ge 1$. Then there exists a constant $a:=a(t,\ell,d)$ such that, if
    \begin{equation}
         \sum_{v\in V(H)} n_v \le \frac{(\log^{(t-2)} k)^k}{(\log^{(t-2)} c)^{c}} \enspace ,
     \label{total-weight}
    \end{equation}
    for some $c\ge 1$ and
    \begin{equation}
         \max\{n_v:v\in V(H)\} \le \frac{(\log^{(t-2)} k)^k}{(\log^{(t-2)} (c+s))^{c+s}} \enspace ,
     \label{max-weight}
    \end{equation}
    where $s := \log c/\log^{(t-1)} c$, then $H$ has an $\ell$-ranking $\varphi:V(H)\to\{1,\ldots, \floor{a(k-c}\}$ such that $\varphi(v)> a(k-\gamma_{t-2,k}(n_v))$ for each $v\in V(H)$.
\end{lem}

\begin{proof}
    Without loss of generality, we may assume that $H$ is edge-maximal with respect to some $r$-rooted $t$-simple tree decomposition $\mathcal{T}:=(B_x:x\in V(T))$.  Let $L_0,\ldots,L_h$ be the BFS layering of $H$ with $L_0:=B_r$.  Note that $h\le\diam(H)\le d$.

    The proof is by induction on $|H|$.  In the base case, $|H|=0$ and there is nothing to prove. Now assume $|H|\ge 1$.  For each subgraph $X$ of $H$, define $n_{X}:=\sum_{v\in V(X)} n_v$ so that \cref{total-weight} implies that $n_H\le(\log^{(t-2)} k)^k/(\log^{(t-2)}c)^c$.
    % For $v\in V(H)$ we use the shorthand $\gamma_v:=\gamma_{t-2,k}(n_v)$ and for any subgraph $X$ of $H$ we use the shorthand $\gamma_X := \gamma_{t-2,k}(n_X)$.  Note that \cref{total-weight,max-weight} imply that $\gamma_H \ge c$ and $\gamma_v\ge c+s$ for each $v\in V(H)$.

    Let
    \begin{equation}
        n_0 := \frac{(\log^{(t-2)} k)^k}{\left(\log^{(t-2)}\left(c+s+\tfrac{\log(c+s)}{\log^{(t-1)}(c+s)}\right)\right)^{c+s+\tfrac{\log(c+s)}{\log^{(t-1)}(c+s)}}} \enspace . \label{heavy-def}
    \end{equation}
    We say that a subgraph $X$ of $H$ is \defin{heavy} if $n_X>n_0$ and $X$ is \defin{light} otherwise.  For a heavy subgraph $X$,
    \begin{equation}
        \frac{n_H}{n_X} < \frac{n_H}{n_0}
        \le \frac{\left(\log^{(t-2)}\left(c+s+\tfrac{\log(c+s)}{\log^{(t-1)}(c+s)}\right)\right)^{c+s+\tfrac{\log(c+s)}{\log^{(t-1)}(c+s)}}}{(\log^{(t-2)} c)^c}
         \in O\left(c^4\right)
    \enspace ,
    \end{equation}
    where the upper bound of $O(c^4)$ is justified by a calculation in \cref{calculation-i}.
    % since $c+s\in O(c)$.

    By \cref{weighted-separator} with the weight function $\xi(v):=n_v$, there exists $S_T\subseteq V(T)$ of size $O(c^4)$ that defines $S:=\bigcup_{x\in S_T} B_x$ such that each component $X$ of $H-S$ is light.  Let $T'$ be the subtree of $T$ induced by $S_T$ and every $T$-ancestor of every node in $S_T$, i.e., $T':=T[\bigcup_{x\in S_T} V(P_T(x))]$. Let $H':=H[\bigcup_{x\in V(T')} B_x]$.
    % Observe that $\mathcal{T'}:=(B_x:x\in V(T'))$ is a $t$-simple tree decomposition of $H'$, so $\stw(H')\le t$.

    For each $i\in\{0,\ldots,h\}$, let $H'_i:=H'[L_i]$. Then  $\mathcal{T}'_i:=(B_x\cap L_i: x\in V(T'))$ is a tree decomposition of $H'_i$ and $H'_i$ is edge-maximal with respect to $\mathcal{T}'_i$.  Each leaf $x$ of $T'$ is an element of $S_T$, therefore $T'$ has at most $|S_T|\in O(c^4)$ leaves. Since $T'$ has $O(c^4)$ leaves, it has $O(c^4)$ branching nodes.  Therefore, by \cref{skeleton-size}, the $(\mathcal{T}_i',\ell)$-skeleton $\hat{H}_i'$ of $H_i'$ has size $|\hat{H_i}'|\in O(c^4)$.  Since $\hat{H}'_i$ is a subgraph of $H'_i$, $\stw(\hat{H}'_i)\le\stw(H'_i)\le\stw(H[L_i])\le t-1$, where the last inequality follows from \cref{simple-bfs-layers}.

    By \weirdref{simple-t-trees}{a} applied to the graph $\hat{H}_i'$ (which has simple treewidth at most $t-1$),\footnote{The case $i=0$ is an exception here, since $H[L_0]=H[B_r]$ is a clique of size at most $t+1$, which certainly has an $\ell$-ranking using at most $t+1\in O(1)$ colours.}
    \[
       \lrn(\hat{H_i}')\in
       O\left(\frac{\log|\hat{H_i}'|}{\log^{(t-1)}|\hat{H_i}'|}\right)
       \subseteq O\left(\frac{\log c^4}{\log^{(t-1)} c^4}\right)
       = O\left(\frac{\log c}{\log^{(t-1)} c}\right)
       = O(s) \enspace .
    \]
    Therefore, by \cref{skeleton-colour} $\lrn(H_i')\in O(s)$, so
    $H_i'$ has an $\ell$-ranking $\varphi:V(H_i')\to \{\floor{a(k-c)}-(i+1)q+1,\ldots,\lfloor a(k-c)\rfloor-iq\}$ for some $q\in O(s)$.

    In the preceding paragraphs, we have defined a colouring $\varphi: V(H')\to \{\floor{a(k-c)}-(h+1)q,\ldots,\floor{a(k-c)})\rfloor\}$. For a sufficiently large constant $a:=a(t,\ell,d)$, $(h+1)q < as$, so $\floor{a(k-c)}-(h+1)q+1 > a(k-c-s)$. Therefore, each vertex in $H'$ receives a colour larger than $\floor{a(k-c-s)}$. By \cref{max-weight}, $\gamma_{t-2,k}(n_v)\ge c+s$ for each $v\in V(H)$, so $\varphi(v)>a(k-c-s) \ge a(k-\gamma_{t-2,k}(n_v))$ for each $v\in V(H')$, as required.

    Since $S':=V(H')\supseteq S$, each component $X$ of $H-V(H')$ is light, so
    \[
       n_{X} \le \frac{(\log^{(t-2)} k)^k}{
        \left(
            \log^{(t-2)}
                \left(
                   c+s+\tfrac{\log(c+s)}{\log^{(t-1)}(c+s)}
               \right)
        \right)^{\left(
           c+s+\tfrac{\log(c+s)}{\log^{(t-1)}(c+s)}
       \right)}
       }
   \]
   Let $c':=c+s$ and let $s':=\log (c+s)/\log^{(t-1)}(c+s)$.  Since each component $X$ of $H-S'$ is light, $n_X$ satisfies \cref{total-weight} with the value $c'+s'\ge c'$ and satisfies \cref{max-weight} with the value $s'$.\footnote{Indeed, $\sum_{v\in V(X)} n_x\le n_0$, so $\max\{n_v:v\in V(X)\}\le n_0= (\log^{(t-2)} k)^k/(\log^{(t-2)}(c'+s'))^{c'+s'}$.}  Therefore, we can apply \cref{t-tree-slack} inductively on $X$ to obtain an $\ell$-ranking $\varphi:V(X)\to\{1,\ldots,\floor{a(k-c')}\}$ in which $\varphi(v)> a(k-\gamma_{t-2,k}(n_v))$ for each $v\in V(X)$, as required.   Doing this for each component $X$ of $H-S'$ completes the colouring $\varphi$ to a total colouring of $H$.

   All that remains is to verify that $\varphi$ is an $\ell$-ranking of $H$. To do this, consider any induced path $u_0,\ldots,u_p$ in $H$ with $p\le\ell$ and $\varphi(u_0)=\varphi(u_p)$.
   % By \cref{induced-paths-only} we may assume that $u_0,\ldots,u_p$ is an induced path in $H$.
   We must show that $\varphi(u_0)<\varphi(u_j)$ for some $j\in\{1,\ldots,p-1\}$. There are a few cases to consider:
   \begin{compactenum}
        \item If $\varphi(u_0)=\varphi(u_p) > a(k-c')$ then $\{u_0,u_p\}\subseteq V(H')$.  By \cref{induced-unimodal}, $x_\mathcal{T}(u_i)$ is a $\mathcal{T}$-ancestor of at least one of $x_\mathcal{T}(u_0)$ or $x_\mathcal{T}(u_p)$ for each $i\in\{0,\ldots,p\}$.  By construction, $T'$ contains every $T$-ancestor of $x_\mathcal{T}(u_0)$ and $T'$ contains every $T$-ancestor of $x_\mathcal{T}(u_p)$.  Therefore $\{u_0,\ldots,u_p\}\subseteq \bigcup_{x\in V(T')} B_x=V(H')$.

        For distinct $i$ and $j$ vertices in $H'_i$ and $H'_j$ receive colours from disjoint sets.  Therefore, since $\varphi(u_0)=\varphi(u_p)$,  $\{u_0,u_p\}\subseteq V(H'_i)$ for some $i\in\{0,\ldots,h\}$.  By \cref{order-relation,induced-unimodal}, $\{u_0,\ldots,u_{p}\}\subseteq \bigcup_{j=0}^{i} V(H'_j)$.  There are two cases to consider:
        \begin{compactenum}
           \item $\{u_0,\ldots,u_{p}\}\subseteq V(H'_i)$ in which case $\varphi(u_j)>\varphi(u_0)$ for some $j\in\{1,\ldots,p-1\}$ since $\varphi$ is an $\ell$-ranking of $H'_i$ (by the application of \cref{skeleton-colour} to $H'_i$); or
           \item $u_j\in V(H'_{i-1})$ for some $j\in\{1,\ldots,p-1\}$.  In this case $\varphi(u_j)\ge\lfloor a(k-c)\rfloor-iq + 1 > \lfloor a(k-c)\rfloor-iq \ge \varphi(u_0)$.
       \end{compactenum}
       \item If $\varphi(u_0)=\varphi(u_p) \le a(k-c')$ then $u_0\in V(X)$ and $u_p\in V(Y)$ for some components $X$ and $Y$ of $H-S'$.  Either
       \begin{compactenum}
            \item $u_j\in S'=V(H')$ for some $j\in\{1,\ldots,p-1\}$ in which case $\varphi(u_j)>a(k-c')\ge\varphi(u_0)$; or
            \item $X=Y$ and $\{u_0,\ldots,u_p\}\subseteq V(X)$, in which case $\varphi(u_j)>\varphi(u_0)$ for some $j\in\{1,\ldots,p\}$ (by the application of \cref{t-tree-slack}, inductively, on $X$). \qedhere
        \end{compactenum}
    \end{compactenum}
\end{proof}

Since our strategy is to use \cref{t-tree-slack} on the first $\ell+1$ BFS layers of $H$ and then recurse on the subgraphs attached to layer $\ell+1$, we need to define vertex weights $n_v$ that allow us to capture the sizes of the subgraphs attached to vertices in layer $\ell+1$.  The following lemma shows that the obvious approach to this does not overcount by more than a factor of $t$.

\begin{lem}\label{size-claim}
    Let $H$ be a graph that is edge-maximal with respect to an $r$-rooted tree decomposition $\mathcal{T}:=(B_x:x\in V(T))$ of width at most $t$ and let $\mathcal{L}:=L_0,\ldots,L_m$ be the BFS layering of $H$ with $L_0=B_r$.  For each $i\in\{0,\ldots,m\}$ and each $v\in L_i$, let $H_v$ be the component of $H[\{v\}\bigcup_{j=i+1}^m L_j]$ that contains $v$ and let $\kappa_v:= t+|H_v|$.  Then $\sum_{v\in L_i} \kappa_v \le t\cdot|\bigcup_{j=i}^m L_j|$.
\end{lem}

\begin{proof}
    For each component $X$ of $H[\bigcup_{j=i+1}^m L_j]$, let $C_X:=L_i\cap N_H(V(X))$.  By \cref{up-neighbours}, $|C_X|\le t$. A vertex $w\in V(X)$ appears in $H_v$ if and only if $v\in C_X$.  Therefore,
    \[
        \sum_{v\in L_i} \kappa_v \le t\cdot|L_i| + \sum_{X}|C_X|\cdot|X| \le t\cdot|L_i| + \sum_{X}t\cdot|X| = t\cdot \left|\bigcup_{j=i}^m L_j\right| \enspace .  \qedhere
    \]
\end{proof}

Finally, we can prove the technical lemma that implies \weirdref{simple-t-trees}{a}.

\begin{lem}\label{t-tree-technical}
    Let $n,t,\ell\in\N\setminus\{0\}$ and $k, c\in\R$ be such that $tn\le (\log^{(t-2)} k)^k/(\log^{(t-2)} c)^{c}$;let $H$ be an $n$-vertex graph that is edge-maximal with respect to some $r$-rooted $t$-simple tree decomposition $\mathcal{T}:=(B_x:x\in V(T))$ of $H$; let $L_0:=\{v_1,\ldots,v_{t'}\}\subseteq B_r$; and let $L_0,\ldots,L_m$ be the BFS layering of $H$.

    Then, there exists an integer $a:=a(t,\ell)$ such that, for any distinct $\phi_0,\ldots,\phi_{t'}\in \{\floor{a(k-c)}+1,\ldots,\floor{ak}\}$ there exists an $\ell$-ranking $\varphi:V(G)\to\{1,\ldots,\floor{ak}\}$ such that
    \begin{compactenum}[(R1)]
        \item $\varphi(v_i)=\phi_i$ for each $i\in\{1,\ldots,t'\}$; and
        \item $\varphi(v)<a(k-c)$ for each $v\in\bigcup_{j=1}^{\ell} L_i$.
    \end{compactenum}
\end{lem}

\begin{proof}
    The proof is by induction on $n$. If $n=0$, then there is nothing to prove.

    Let $n_0:=(\log^{(t-2)} k)^k/(\log^{(t-2)} (c+s))^{c+s}$ and, for each $v\in V(H)$, let $\kappa_v$ be defined as in \cref{size-claim}.  We say that a vertex $v\in L_{\ell+1}$ is \defin{dangerous} if $\kappa_v>n_0$ and $v$ is \defin{harmless} otherwise.

    We now assign weights to the vertices of the graph $H_0:=H[\bigcup_{j=0}^{\ell+1}]$ in such a way that we can apply \cref{t-tree-slack} to $H_0$.  For each $v\in\bigcup_{j=0}^\ell L_j$, we set $n_v:=1$.  For each $v\in L_{\ell+1}$, we set $n_v := \min\{n_0, \kappa_v\}$.  With this assignment of weights, \cref{size-claim} implies that $\sum_{v\in V(H_0)} n_v\le tn \le (\log^{(t-2)} k)^k/(\log^{(t-2)} c)^c$, which satisfies \cref{total-weight} and, by definition, $\max\{n_v:v\in V(H_0)\}\le n_0$ which satisfies \cref{max-weight}.

    In the following, we use the shorthand $\gamma_v := \gamma_{t-2,k}(n_v)$.   By \cref{t-tree-slack}, $H_{0}$ has an $\ell$-ranking $\varphi:V(H_0)\to\{1,\ldots,\lfloor a(k-c)\rfloor\}$ in which $\varphi(v)> a(k-\gamma_v)$ for each $v\in V(H_0)$.
    By \cref{size-claim}, the number of dangerous vertices in $L_{\ell+1}$ is at most
    \[
        \frac{tn}{n_0} \in O\left(\frac{(\log^{(t-2)} (c+s))^{c+s}}{(\log^{(t-2)} c)^c}\right) \in O(c) \enspace ,
    \]
    where the $O(c)$ upper bound is justified by a calculation in \cref{calculation-ii}.
    Before continuing, we make the following modifications to $\varphi$.

    \begin{compactenum}
        \item We set $\varphi(v_i):=\phi_i$ for each $i\in\{1,\ldots,t'\}$.
        \item For each dangerous vertex $v$, we set $\varphi(v)$ to a distinct value in $\{\lfloor a(k-c)\rfloor+1,\ldots,ak\}\setminus\{\phi_1,\ldots,\phi_t\}$. (Since the number of dangerous vertices is $O(c)$, this is always possible.)
    \end{compactenum}
    These modifications ensure that $\varphi$ satisfies requirements (R1) and (R2) and, since they only introduce new unique colours larger than any existing colour, they preserve the fact that $\varphi$ is an $\ell$-ranking of $H_0$.

    For each component $X$ of $H-V(H_0)$, let $C_X:=L_{\ell+1}\cap N_H(V(X))$ and let $H_X:=H[C_X\cup V(X)]$.  By \cref{up-neighbours}, $|C_X|\le t$. We apply induction on $H_X$ for each component $X$ of $H-H_0$ using colours $\phi_1',\ldots,\phi_{t'}'$ already assigned to the vertices in $C_X$.
    When we do this, we obtain an $\ell$-ranking of $H_X$ in which each vertex $w$ of $X[\bigcup_{j=\ell+2}^{2\ell+1} L_j]$ receives a colour $\varphi(w) \le a(k-\gamma_{t-2,k}(|H_X|))$.

    For each harmless vertex $v\in C_X$, $X$ is a subgraph of $H_v$, so $n_v\ge t+|X|\ge |C_X|+|X|= |H_X|$, so $\gamma_v \le \gamma_{t-2,k}(|H_X|)$. Therefore, for each harmless $v\in C_X$, $\varphi(v)> a(k-\gamma_v)\ge\varphi(w)$ for each $w\in V(X)\cap [\bigcup_{j=\ell+2}^{2\ell+1} L_j]$.  For each dangerous vertex $v\in C_X$, $\varphi(v)>a(k-c)$.  Since $|H_X|\le|H|$, $\gamma_{t-2,k}(|H_X|) \ge c$.  Therefore each dangerous vertex $v\in C_X$ also receives a colour larger than each vertex $w$ in $X[\bigcup_{j=\ell+2}^{2\ell+1} L_j]$.

    All that remains is to verify that the resulting colouring is, indeed, an $\ell$-ranking of $H$.  Consider some induced path $u_0,\ldots,u_p$ in $H$ of length $p\le\ell$ such that $\varphi(u_0)=\varphi(u_p)$.  There are some cases to consider:

    \begin{compactenum}
        \item $\{u_0,u_p\}\subseteq V(H_0)$. In this case, \cref{order-relation,induced-unimodal} imply that $\{u_0,\ldots,u_p\}\subseteq V(H_0)$.  However, we have already established that $\varphi$ is an $\ell$-ranking of $H_0$ through the application of \cref{t-tree-slack} and the subsequent recolouring of vertices in $L_0$ and $L_{\ell+1}$.  Therefore, $\varphi(u_0)<\max\{\varphi(u_1),\ldots,\varphi(u_{p-1})\}$.

        \item $u_0\in V(X)$ for some component $X$ of $H-V(H_0)$ and $u_i\in C_X$ for some $i\in\{1,\ldots,p-1\}$.  Since $i<p\le\ell$, this implies that $u_0\in\bigcup_{j=\ell+2}^{\ell+p+1} L_j\subseteq \bigcup_{j=\ell+2}^{2\ell+1}L_j$.  We have already argued above that this implies that $\varphi(u_i)>\varphi(u_0)$.

        \item $\{u_0,\ldots,u_p\}\subseteq V(X)$ for some component $X$ of $H-V(H_0)$.  In this case, the inductive hypothesis ensures that $\varphi$ is an $\ell$-ranking of $X$, so $\varphi(u_0)<\max\{\varphi(u_0),\ldots,\varphi(u_p)\}$. \qedhere
    \end{compactenum}
\end{proof}

Rewriting \cref{t-tree-technical} in terms of $n$ yields \weirdref{simple-t-trees}{a}:

\begin{proof}[Proof of \cref{simple-t-trees} (upper bound)]
    When $t=1$, $H$ is a collection of vertex-disjoint paths and $\lrn(H)\in O(\log\ell)=O(1)=O(\log n/\log^{(1)} n)$ (see \cref{logarithmic-base-case}).  Assume now that $t\ge 2$.
    Fix some $\epsilon > 0$, let $k:=(1+\epsilon)\log (tn)/\log^{(t)} n$, and let $c:=\tau(t-2)$.
    By \cref{t-tree-technical}, $\lrn(H)\in O(k)$ provided that $k$ satisfies
    \[  \frac{(\log^{(t-2)} k)^{k}}{(\log^{(t-2)} c)^c} \ge tn
        % \Leftrightarrow
        % k \ge \frac{\log n}{\log^{(t-1)} k}
        \Leftrightarrow
        \frac{k\log^{(t-1)} k}{\log (tn)} \ge 1 \enspace ,
    \]
    for all sufficiently large $n$.
    With our choice of $k$ we have
    \begin{align*}
      \frac{k\log^{(t-1)} k}{\log (tn)}
      % & = \frac{(1+\epsilon)\log^{(t-1)} k}{\log^{(t)} n} \\
      & = (1+\epsilon)\cdot\frac{\log^{(t-1)}\left((1+\epsilon)\log n/\log^{(t)} n\right)}{\log^{(t)} n} \\
      & = (1+\epsilon)\cdot\frac{\log^{(t-2)}\left(\log^{(2)} n + \log(1+\epsilon) - \log^{(t+1)} n\right)}{\log^{(t-2)}\left(\log^{(2)} n\right)} \\
      & = (1+\epsilon)\cdot\frac{\log^{(t-2)}\left(\log^{(2)} n-o(\log^{(2)} n)\right)}{\log^{(t-2)}\left(\log^{(2)} n\right)} \\
      & \qquad \text{(since $t\ge 2$, so $\log^{(t+1)} n\in o(\log^{(2)} n)$ )}\\
      % & = (1+\epsilon)\cdot\frac{\log^{(t-3)}\left((\log^{(3)} n)+\log(1-o_n(1))\right)}{\log^{(t-3)}\left(\log^{(3)} n\right)} \\
      % & = (1+\epsilon)\cdot\frac{\log^{(t-3)}\left(\log^{(3)} n+o_n(1)\right)}{\log^{(t-3)}\left(\log^{(3)} n\right)} \\
      & \rightarrow 1+\epsilon
    \end{align*}
    as $n\rightarrow\infty$.\footnote{If there exists some $\epsilon >0$ and $x_0$ such that $f(x)-\delta x\le f(x-\delta x)\le f(x)$ for all $x\ge x_0$ and all $\delta\in[0,\epsilon]$ then $\lim_{x\to\infty} [f(x-o(x))/f(x)] = 1$.  Here we are using this with $f(x):=\log^{(t-2)} x$ and $x:=\log^{(2)} n$.}
    Therefore, $\lrn(H)\in O(k)=O(\log n/\log^{(t)} n)$, as required.
    % \[
    %
    % \]
    % \[
    %     \frac{2\log n}{\log^{(t)} n}
    %     \ge 2\frac{\log n}{\log^{(t-1)}(\log n -\log^{(t)} n)}
    %     = \frac{\log n}{\log^{(t)} n - O(1/\log n)}
    %     = \frac{\log n}{\log^{(t)} n}+O(1) \enspace .
    % \]
    % which is clearly true for all sufficiently large $n$.\todo{Correct this.}
\end{proof}

\subsection{Bounded Genus Graphs}

The upper bound in \cref{bounded-genus} for bounded genus graphs follows from \cref{simple-t-trees,product-lemma,dumb} and the following recent result of \citet{distel.hickingbotham.ea:improved}:

\begin{thm}[\cite{distel.hickingbotham.ea:improved}] \label{product-structure-genus}
    For every $n$-vertex graph $G$ of Euler genus at most $g$, there exists some at most $n$-vertex simple $3$-tree $H$ and some path $P$ such that $G$ is isomorphic to a subgraph of $H\boxtimes K_{\max\{2g,3\}}\boxtimes P$
\end{thm}

\subsection{Other Graph Families with Product Structure}

As noted in the introduction, several other families of graphs are known to have product structure theorems like \cref{product-structure,product-structure-genus}.  In particular, \citet{dujmovic.joret.ea:planar} show:

\begin{thm}[\cite{dujmovic.joret.ea:planar}]\label{apex-minor-free}
    For any apex graph $A$, there exists a value $t$ such that any $n$-vertex $A$-minor free graph $G$ is isomorphic to a subgraph of $H\boxtimes P$ where $|H|\le n$, $\tw(H)\le t$, and $P$ is a path.
\end{thm}

A graph is \defin{$(g,k)$-planar} if it has an embedding in a surface of Euler genus $g$ in which each edge is involved in at most $k$ crossings with other edges.  \citet{dujmovic.morin.ea:structure} prove analogues of \cref{apex-minor-free} for some non-minor-closed families of graphs, the most well-known of which are the $(g,k)$-planar graphs:

\begin{thm}[\cite{dujmovic.morin.ea:structure}]\label{gk-planar}
    For any integers $g$ and $k$, there exists a value $t$ such that any $n$-vertex $(g,k)$-planar graph $G$ is isomorphic to a subgraph of $H\boxtimes P$, where $|H|\le n$, $\tw(H)\le t$, and $P$ is a path.
\end{thm}

\begin{proof}[Proof of \cref{meta}]
    For any $n$-vertex member $G$ of these graph families, \cref{gk-planar,apex-minor-free} show that $G$ is a subgraph of $H\boxtimes P$ with $\tw(H)\le t$.  \Cref{t-trees,product-lemma,dumb} then imply \cref{meta}.
\end{proof}

\subsection{Dependence on $\ell$}
\label{dependence-on-ell}

Throughout this section, we have assumed that $\ell$ and $t$ were fixed constants, independent of $n$.  We now describe the dependence of our results on $\ell$.  Since all of our upper bounds are based on \weirdref{simple-t-trees}{a} we begin by discussing \weirdref{simple-t-trees}{a} and its proof, which is the subject of \cref{simple-treewidth-section}.  We will show that, for fixed constant $t$, the bound for \weirdref{simple-t-trees}{a} is easily shown to be $\lrn(G)\in O(\ell^{t-1}\log^t\ell \log n/\log^{(t)} n)$.

Recall that the overall structure of the proof is by induction on $t$ with the base case $t=1$.  The case $t=1$ is described in \cref{logarithmic-base-case}, which explains how a simple $1$-tree (a collection of disjoint paths) has an $\ell$-ranking using $O(\log\ell)$ colours.  This establishes the result for the base case.

% The first part of the proof that makes use of the assumption that $\ell$ is fixed is \cref{skeleton-colour} which claims a bound of $\lrn(H)\in O(\lrn(\hat{H}))$ but actually proves the more precise bound $\lrn(H)\in O(\ell t+ \lrn(\hat{H}))=O(\ell+\lrn(\hat{H}))$ for constant $t$.

The first place in which $\ell$ is treated as a constant is in \cref{t-tree-slack}, in which the constant $a:=a(t,\ell,d)$ appears.  The only place \cref{t-tree-slack} is used is in the proof of \cref{t-tree-technical}, where it is applied with $d\in O(\ell)$.  Under these conditions, taking $a\in O(\ell^{t-1}\log^t \ell)$ is sufficient, as we now show. Within the proof of \cref{t-tree-slack}, \weirdref{simple-t-trees}{a} is used on the graph $\hat{H}'_i$ for each $i\in\{0,\ldots,h\}$ to show that $\lrn(\hat{H}'_i)\in O(s)$.  Here $\hat{H}'_i$ is a treewidth $t-1$ graph with $O(\ell^t c^4)$ vertices and $s=\log c/\log^{(t-1)} c$.  With the more precise inductive hypothesis, this becomes
\begin{align*}
    \lrn(\hat{H}'_i) & \in O\left(\ell^{t-2}\log^{t-1}\ell
    \cdot \frac{\log (\ell^{t} c^4))}{\log^{(t-1)} (\ell^t c^4)}\right) \\
    & = O\left(\ell^{t-2}\log^{t-1}\ell\cdot \frac{\log c + \log\ell}{\log^{(t-1)} c}\right) \\
    & \subseteq O\left(\ell^{t-2}\log^{t-1}\ell\cdot \frac{(\log c)(\log\ell)}{\log^{(t-1)} c}\right) \\
    & = O((\ell^{t-2}\log^t\ell)s)
\end{align*}
Doing this for each $i\in\{0,\ldots,h\}$ gives a colouring of $H'$ using the colour set $\{\lfloor a(k-c)-(h+1)q+1,\ldots,\lfloor a(k-c)\rfloor\}$ for some $q\in O((\ell^{t-2}\log^t\ell)s)$.  Here $h\in O(d)=O(\ell)$. By choosing a sufficiently large $a\in O(\ell^{t-1}\log^t\ell)$, so that $as > (h+1)q$, this colouring uses only colours from the set $\{\lfloor a(k-c-s)+1,\ldots,\lfloor a(k-c-1)\rfloor$.  The rest of the proof applies \cref{t-tree-slack} inductively on each of the uncoloured components of $H-V(H')$ to complete the colouring using smaller colours in the set $\{1,\ldots,\lfloor a(k-c-s)\rfloor\}$ and is unchanged.  The remainder of the proof is unchanged and proves the following refinement of \cref{simple-t-trees}:

\begin{itemize}
    \item[(\ref{simple-t-trees})] For any fixed integer $t\ge 1$ and every integer $\ell\ge 1$, every $n$-vertex graph $H$ of simple treewidth at most $t$ has $\lrn(H)\in O(\ell^{t-1}\log^t \ell\log n/\log^{(t)} n)$.
\end{itemize}

Using this refinement of \cref{simple-t-trees} gives the following refined versions of \cref{planar,bounded-genus,t-trees,meta}.
\begin{itemize}
    \item[(\ref{planar})] For every integer $\ell\ge 1$, every $n$-vertex planar graph $G$ has $\lrn(G)\in O((\ell\log \ell)^3\log n/\log^{(3)} n)$. (The additional factor of $\ell$ comes from the application of \cref{product-lemma} on the graph $H\boxtimes P\supseteq G$, where $\stw(H)\le 3$.)

    \item[(\ref{t-trees})] For any fixed integer $t\ge 0$ and every integer $\ell\ge 1$, every $n$-vertex graph $H$ of treewidth at most $t$ has $\lrn(H)\in O(\ell^t\log^{t+1}\ell \log n/\log^{(t+1)}n)$.

    \item[(\ref{bounded-genus})] For any integers $g\ge 0$, $\ell\ge 1$, every  $n$-vertex graph $G$ of Euler genus at most $g$ has $\lrn(G)\in O((\ell\log \ell)^3 g\log n/\log^{(3)} n)$.

    \item[(\ref{meta})] For each of the following graph classes $\mathcal{G}$:
    \begin{compactenum}
        \item the class of graphs excluding a particular apex graph $A$ as a minor; and
        \item the class of $(g,k)$-planar graphs,
    \end{compactenum}
    there exists an integers $c=c(\mathcal{G})$ and $b=b(\mathcal{G})$ such that every $n$-vertex graph $G\in\mathcal{G}$ has $\lrn(G)\in O(b\ell^{c-1}\log^c\ell\log n/\log^{(c)} n)$.
\end{itemize}

%========================================================================
\section{Discussion}
\label{conclusion}

We have given asymptotically optimal bounds on the number of colours required by $\ell$-rankings of $n$-vertex graphs of treewidth $t$, graphs of simple treewidth $t$, planar 3-trees, outerplanar graphs, and planar graphs.  Prior to this work, the best known bounds for planar graphs were $\Omega(\log n/\log\log n)$ (trees) and $O(\log n)$.

Our upper bounds are constructive and lead to straightforward linear time algorithms for finding $\ell$-rankings of (simple) treewidth $t$ graphs, including planar 3-trees, and outerplanar graphs.  For a planar graph $G$ we can use the recent linear time algorithm of \citet{bose.morin.ea:optimal} for finding the simple $3$-tree $H$ and the path $P$ such that $G\subseteq H\boxtimes K_3\boxtimes P$ (\cref{product-structure}) to find an $\ell$-ranking of $G$ in $O(n)$ time.

% Two directions for future work stand out:
% \begin{inparaenum}[(i)]
    % \item This work has deliberately ignored constants that depend on $\ell$ and $t$.  In the current proof, the hidden constant in the big-Oh notation is at least $\Omega(\ell^t)$; see \cref{path-induced}.  Can this be improved?
For constant $d$, the lower and upper bounds for 2-ranking $d$-degenerate graphs are $\Omega(n^{1/3})$ and $O(\sqrt{n})$, respectively.  Closing this gap is an intriguing open problem.
% \end{inparaenum}

\section*{Acknowledgement}

The authors are grateful to David Wood who, after reading an early draft of this paper, pointed us to the notion of simple treewidth, which greatly simplifies and unifies the exposition of our results.

\bibliographystyle{plainnat}
\bibliography{us}

\appendix
\section{Calculations}
\label{calculation}

\subsection{Calculation in the Proof \cref{t-tree-slack}}
\label{calculation-i}

\begin{align*}
    &\hspace{-2em} \left(\log^{(t-2)}\left(c+s+\tfrac{\log(c+s)}{\log^{(t-1)}(c+s)}\right)\right)^{c+s+\tfrac{\log(c+s)}{\log^{(t-1)}(c+s)}} \\
    & = \left(\log^{(t-2)}c\right)^{\left(c+s+\tfrac{\log(c+s)}{\log^{(t-1)}(c+s)}\right)
        \left(\tfrac{\log^{(t-1)}(c+s+\log(c+s)/\log^{(t-1)}(c+s))}{\log^{(t-1)}c}\right)} \\
        & \qquad\qquad\text{(change of base)} \\
    & < (\log^{(t-2)}c)^{\left(c+s+\tfrac{\log(c+s)}{\log^{(t-1)}(c+s)}\right)
        \left(1 + \tfrac{s+\log(c+s)/\log^{(t-1)}(c+s)}{\prod_{j=0}^{t-1}\log^{(j)}c}\right)} \\
        & \qquad\qquad\text{(by \cref{logi-ratio})}\\
    & < (\log^{(t-2)}c)^{\left(c+s+\tfrac{\log(c+s)}{\log^{(t-1)}(c+s)}\right)
        \left(1 + \tfrac{2\log(c+s)/\log^{(t-1)}(c+s)}{\prod_{j=0}^{t-1}\log^{(j)}c}\right)} \\
        & \qquad\qquad\text{(since $s=\log c/\log^{(t-1)} c$), so $s<\log(c+s)/\log^{(t-1)}(c+s)$)}\\
    & < (\log^{(t-2)}c)^{\left(c+s+\tfrac{\log(c+s)}{\log^{(t-1)}(c+s)}\right)
        \left(1 + \tfrac{2\log(c+s)/\log^{(t-1)}c}{\prod_{j=0}^{t-1}\log^{(j)}c}\right)} \\
        & \qquad\qquad\text{(since $c+s>c$)}\\
    & = (\log^{(t-2)}c)^{\left(c+s+\tfrac{\log(c+s)}{\log^{(t-1)}(c+s)}\right)
        \left(1 + \tfrac{2\log(c+s)}{c\log c\cdot\left(\prod_{j=2}^{t-1}\log^{(j)}(c+s)\right)\log^{(t-1)}c}\right)} \\
        & \qquad\qquad\text{(since $t\ge 2$)}\\
    & \le (\log^{(t-2)}c)^{\left(c+s+\tfrac{\log(c+s)}{\log^{(t-1)}(c+s)}\right)
        \left(1 + \tfrac{2\log(c+s)}{c\log c\cdot\log^{(t-1)}c}\right)} \\            & \qquad\qquad\text{(since $c\ge \tau(t-1)$, so $\textstyle \prod_{j=2}^{t-1}\log^{(j)}c\ge 1$)} \\
    & \le (\log^{(t-2)}c)^{\left(c+s+\tfrac{\log c+s/c}{\log^{(t-1)}(c+s)}\right)
        \left(1 + \tfrac{2\log c+2s/c}{c\cdot\log^{(t-1)}c}\right)} \\            & \qquad\qquad\text{(by \cref{log-x-plus-a})} \\
    & \le (\log^{(t-2)}c)^{c+s+\tfrac{\log c}{\log^{(t-1)}c} +
        \tfrac{2\log c}{\log^{(t-1)}c} + o\left(\tfrac{1}{\log^{(t-1)} c}\right)}  \\
    & = (\log^{(t-2)}c)^{c+\tfrac{4\log c}{\log^{(t-1)}c} + o\left(\tfrac{1}{\log^{(t-1)} c}\right)}  \\
    % & = (\log^{(t-2)}c)^{c}\cdot O(c^4) \\
    & = (1+o_c(1))\cdot c^4\cdot \left(\log^{(t-2)}c\right)^{c} \\
    & = O\left(c^4\cdot\left(\log^{(t-2)}c\right)^{c}\right) \enspace .
\end{align*}

Therefore
\[
  \frac{\left(\log^{(t-2)}\left(c+s+\tfrac{\log(c+s)}{\log^{(t-1)}(c+s)}\right)\right)^{c+s+\tfrac{\log(c+s)}{\log^{(t-1)}(c+s)}}}
  {(\log^{(t-2)} c)^c}
  \in O(c^4) \enspace .
\]

\subsection{Calculation in the Proof \cref{t-tree-technical}}
\label{calculation-ii}

\begin{align*}
    % \frac{\left(\log^{(t-2)} k\right)^k}{n_{0}}
    & \hspace{-2em}(\log^{(t-2)}(c+s))^{c+s} \\
    & = \left(\log^{(t-2)}\left(c+\tfrac{\log c}{\log^{(t-1)} c}\right)\right)^{c+\tfrac{\log c}{\log^{(t-1)} c}} \\
    & = \left(\log^{(t-2)}c\right)^{\left(c+\tfrac{\log c}{\log^{(t-1)} c}\right)
    \left(\tfrac{\log^{(t-1)}\left(c+\log c/\log^{(t-1)} c\right)}{\log^{(t-1)} c}\right)} &\text{(change of base)} \\
    & = \left(\log^{(t-2)}c\right)^{\left(c+\tfrac{\log c}{\log^{(t-1)} c}\right)
    \left(
    1 + \tfrac{\log c/\log^{(t-1)} c}{\prod_{j=0}^{t-1}\log^{(j)}c}
    \right)} & \text{(by \cref{logi-ratio})} \\
    & = \left(\log^{(t-2)}c\right)^{\left(c+\tfrac{\log c}{\log^{(t-1)} c}\right)
    \left(
    1 + \tfrac{1}{c\cdot\left(\prod_{j=2}^{t-1}\log^{(j)}c\right)\cdot\log^{(t-1)} c}
    \right)} & \text{(for $t\ge 2$)} \\
    & \le \left(\log^{(t-2)}c\right)^{\left(c+\tfrac{\log c}{\log^{(t-1)} c}\right)
    \left(
    1 + \tfrac{1}{c\cdot\log^{(t-1)} c}
    \right)} & \text{($c\ge \tau(t-1)$, so $\prod_{j=2}^{t-1}\log^{(j)}c\ge 1$)} \\
    & = \left(\log^{(t-2)}c\right)^{\left(c+\tfrac{\log c}{\log^{(t-1)} c}+
    \tfrac{1}{\log^{(t-1)} c} + \tfrac{\log c}{c\cdot(\log^{(t-1)} c)^2}
    \right)}  \\
    & = \left(\log^{(t-2)}c\right)^{\left(c+\tfrac{\log c}{\log^{(t-1)} c}
      + O_c\left(\tfrac{1}{\log^{(t-1)} c}\right)
    \right)}  \\
    & \in O\left(c\cdot\left(\log^{(t-2)}c\right)^c\right)
\end{align*}

Therefore
\[
    \frac{(\log^{(t-2)} (c+s))^{c+s}}{(\log^{(t-2)}c)^c} \in O(c) \enspace .
\]

\end{document}